\documentclass[10pt,letterpaper]{article}
\usepackage[top=0.85in,left=2.75in,footskip=0.75in]{geometry}
\usepackage{amsmath,amssymb,mathtools,nicefrac}
\usepackage{changepage}
\usepackage{textcomp,marvosym}
\usepackage{cite}
\usepackage{nameref,hyperref}
\usepackage[nopatch=eqnum]{microtype}
\DisableLigatures[f]{encoding = *, family = * }
\usepackage[table]{xcolor}
\usepackage{array}
\usepackage{tikz,standalone,graphicx}
\newcolumntype{+}{!{\vrule width 2pt}}
\newlength\savedwidth

\usepackage{enumitem}

\raggedright
\usepackage[aboveskip=1pt,labelfont=bf,labelsep=period,justification=raggedright,singlelinecheck=off]{caption}

\makeatletter
\renewcommand{\@biblabel}[1]{\quad#1.}
\makeatother

\usepackage{empheq}

\usepackage{lastpage,fancyhdr,graphicx}
\usepackage{epstopdf}
\fancyheadoffset[L]{2.25in}
\fancyfootoffset[L]{2.25in}
\usepackage{amsthm}

\newtheorem{theorem}{Theorem}

\newtheorem{lemma}[theorem]{Lemma}

\newtheorem{assumption}[theorem]{Assumption}
\newtheorem{remark}[theorem]{Remark}

\newcommand{\R}{{\mathbb{R}}}

\DeclareFontFamily{U}{mathb}{\hyphenchar\font45}
\DeclareFontShape{U}{mathb}{m}{n}{
   <5> <6> <7> <8> <9> <10> gen * mathb
   <10.95> mathb10
   <12> <14.4> <17.28> <20.74> <24.88> mathb12
}{}
\DeclareSymbolFont{mathb}{U}{mathb}{m}{n}
\DeclareMathSymbol{\ggcurly}{3}{mathb}{"CF}
\DeclareMathSymbol{\llcurly}{3}{mathb}{"CE}

\let\succeq=\succcurlyeq
\let\preceq=\preccurlyeq
\DeclareMathOperator{\interior}{int}
\begin{document}

\vspace*{0.2in}

\begin{flushleft}
  {\Large
    \textbf\newline{A mathematical model for depression and resilience}
  }
  \newline
  \\

  Bj\"orn S.~R\"uffer\textsuperscript{1\Yinyang*},
  Michael~Sch\"onlein\textsuperscript{1\Yinyang}
  \\
  \bigskip
  \textbf{1} Chair of Applied Mathematics, Bauhaus-Universität~Weimar, Weimar, Germany
  \\
  \bigskip

  \Yinyang These authors contributed equally to this work.

  * E-mail: bjoern.rueffer@uni-weimar.de

\end{flushleft}

\section*{Abstract}
A basic dynamical model for (clinical) depression is presented to
describe the time evolution of two coupled states: a resilience level
and a depression symptom. The resilience level can also be interpreted
in terms of the memory of past symptoms. The model consists of a
system of two coupled first order differential equations without free
parameters that qualitatively captures different courses of illness,
without the overhead of a derivation from neuroscientific first
principles. A comprehensive mathematical analysis of the model is
provided, including equilibria, stability properties, and
monotonicity. The model can reproduce chronic, delayed, recovery, and
resilience scenarios that are prominent in the literature, as well as
scenarios such as improvement from pre-existing conditions, burnout
from low-grade adversity, or isolated and recurrent depressive
episodes. Supplementary to the manuscript are computational tools to
replicate the figures and, without prior mathematical or programming
skills, to experiment with the model interactively in the web browser
(at \texttt{https://rsmodel.org/}).

\clearpage
\newgeometry{top=0.85in,left=1in,right=1in,footskip=0.75in}

\section*{Introduction}

According to the World Health Organization (WHO), depression is a
common mental disorder affecting 5.7\% of the adult population and a
leading cause of disability worldwide, and can lead to suicide. The
WHO further estimates that over 700\,000 people die due to suicide
every year, with suicide being the third leading cause of death in
15--29~year-olds~\cite{world-health-organization2023-depressive-disorder-depression}. 

Existing mathematical models of depression can broadly be divided into
symptom network models, statistical prediction models, agent-based
simulations, biologically detailed mechanistic models, and dynamical
systems approaches. While dynamical systems have been proposed to
explain resilience, tipping points, and recovery, comparatively few
studies formulate analytically tractable low-dimensional ordinary
differential equation (ODE) models that explicitly describe the
reciprocal interaction between depressive symptoms and resilience.

In neuroscience and biology, research such as
\cite{byrumahearnkrishnan1999-a-neuroanatomic-model-for-depression,
  disnerbeevershaighbeck2011-neural-mechanisms-of-the-cognitive-model-of-depression},
features causality networks that show how signaling pathways and areas
in the brain influence each other, or how schemas, memories, triggers
and behaviors are (statically) coupled in different psychological
theories, while others put emphasis on more dynamic temporal aspects
of
depression~\cite{demiccheng2014-modeling-the-dynamics-of-disease-states-in-depression,
  chengdorsognachou2020-mathematical-modeling-of-depressive-disorders:-circadian-driving-bistability-and-dynamical-transitions,
  bothhoogendoornkleintreur2008-modeling-the-dynamics-of-mood-and-depression,
  tuckwellmiura1978-a-mathematical-model-for-spreading-cortical-depression}.

Following the idea of contemporary resilience research that relates a
decline in resilience to the accumulation of stress over
time, such as~\cite{liebrigottischafer2025-ausgebrannt-sein:-burnout-als-risikozustand,
  arnoldschilbachrigotti2023-paradigmen-der-psychologischen-resilienzforschung, galatzer-levyhuangbonanno2018-trajectories-of-resilience-and-dysfunction-following-potential-trauma:-a-review-and-statistical-evaluation, schaferkunzlerkalischtuscherlieb2022-trajectories-of-resilience-and-mental-distress-to-global-major-disruptions,bonanno2004-loss-trauma-and-human-resilience:-have-we-underestimated-the-human-capacity-to-thrive-after-extremely-aversive-events},
our model focuses on the interaction between depression symptoms
and resilience, modeling both of these as coupled dynamic quantities
in the most basic coupling topology. Our model is simple in terms of
its mathematical description, yet rich enough to exhibit behaviors
consistent with anecdotal observations about
the course of a depression, cf.~Fig~\ref{fig:scenarien}.
\begin{figure}[htb]
  \centering
  \includegraphics[width=0.8\linewidth]
  {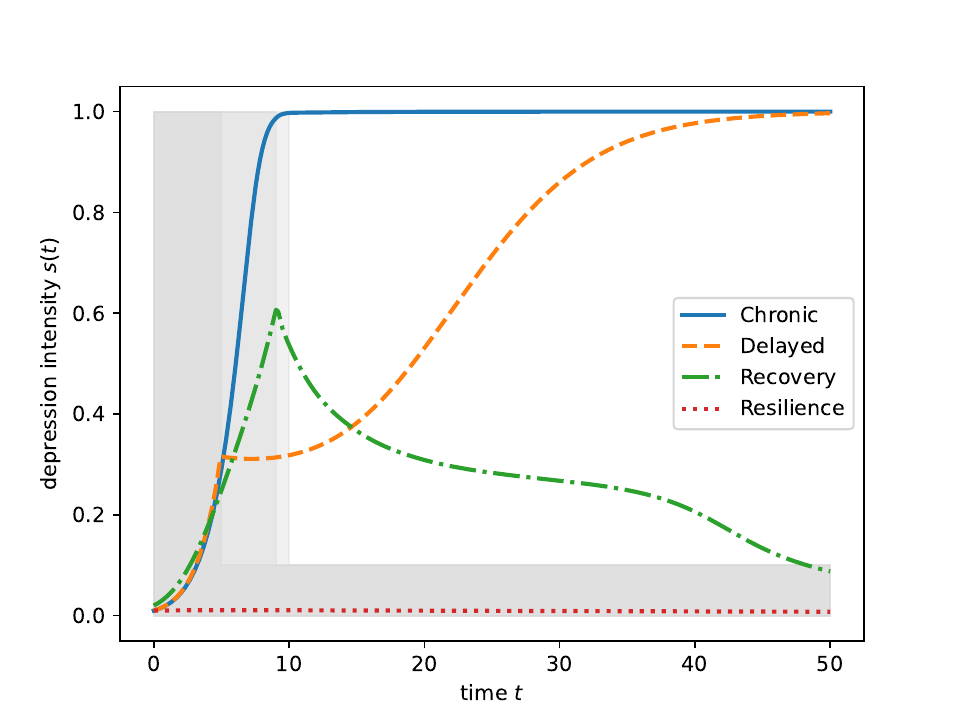}
  \caption{%
    \label{fig:scenarien}
    \textbf{Scenarios.}
    Characteristic courses of depression symptom $s(t)$ over
    time following external stress (shaded region). These quantitative
    trajectories are generated by our model and resemble the
    (qualitative) trajectories showcased in
    \cite{galatzer-levyhuangbonanno2018-trajectories-of-resilience-and-dysfunction-following-potential-trauma:-a-review-and-statistical-evaluation,
    schaferkunzlerkalischtuscherlieb2022-trajectories-of-resilience-and-mental-distress-to-global-major-disruptions,
    liebrigottischafer2025-ausgebrannt-sein:-burnout-als-risikozustand},
    which in turn are based on
    \cite{bonanno2004-loss-trauma-and-human-resilience:-have-we-underestimated-the-human-capacity-to-thrive-after-extremely-aversive-events}. The
    shaded area indicates adverse input levels, cf.\
    Figs~\ref{fig:chronic-scenario}--\ref{fig:resilience-scenario}
    for details. See Fig~\ref{fig:alternative-scenarios} for another
    version of this diagram resembling the figures in
    \cite{galatzer-levyhuangbonanno2018-trajectories-of-resilience-and-dysfunction-following-potential-trauma:-a-review-and-statistical-evaluation,
      schaferkunzlerkalischtuscherlieb2022-trajectories-of-resilience-and-mental-distress-to-global-major-disruptions,
      liebrigottischafer2025-ausgebrannt-sein:-burnout-als-risikozustand}
    even closer.}
\end{figure}
Loosely speaking, during stress the course of depression worsens
faster if a person is less resilient, while resilience levels deplete
over time the longer and the more severely a person remains
depressed. One state in our model represents the severity of
depression or another suitable symptom that is a good indicator of
depression. This is something that would commonly be measured by a
practitioner using a questionnaire, such as the Beck Depression
Inventory-Second Edition~(BDI-II), a widely used 21-item self-report
inventory measuring the severity of depression in adolescents and
adults~\cite{becksteerbrownothers1996-manual-for-the-beck-depression-inventory-ii},
or the Beck Hopelessness Scale~(BHS), an instrument for assessing
cognitive thoughts among suicidal
persons~\cite{becksteer1988-bhs-beck-hopelessness-scale:-manual}.
Alternatively, the first state could be a quantitative measure of the
level of a stress hormone often associated with depression, such as
cortisol, adrenocorticotropic hormone (ACTH), epinephrine,
norepinephrine, or corticotropin-releasing hormone (CRH).

The second state in our model measures a patient's level of resilience, with higher
levels of resilience being associated with a slower onset or adoption
of depression symptoms in response to external negative experiences or
stimuli. While this resilience state would commonly be unavailable to
direct measurement, it could be inferred from measurements of the
depression symptom and external inputs over time. As we will discuss
in more detail later, the resilience level can alternatively be
interpreted in terms of the memory of past depression experiences, or,
in other words, the accumulation of stress. A conceptual diagram of
the model is given in Fig~\ref{fig:blockdiagram}.
\begin{figure}[htb]
  \centering
  \includegraphics[width=.9999\linewidth]
  {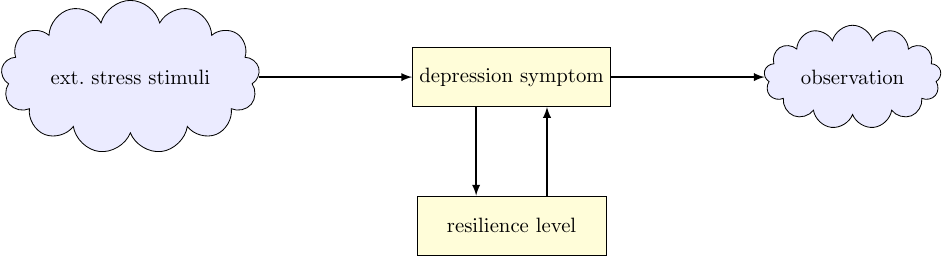}
  \vskip 1ex
  \caption{%
    \label{fig:blockdiagram}%
    \textbf{Block diagram.}
    In our model the depression symptom and resilience level
    are coupled dynamic quantities. Commonly only the depression level
    would be observed by the practitioner.}
\end{figure}

\section*{Mathematical model derivation and description}

In this section we present our mathematical model
describing the interplay between the symptom of a depression and the
resilience of a person, both as dynamic quantities, and in temporal
response to external negative influences.
In order to keep the model ``as simple as possible, but not simpler'',
we make a few assumptions. In the following the variable $t$
refers to time and belongs to a real time interval, usually
$\R_{+}\coloneqq [0,\infty)$.

\begin{assumption}
  At every time instant $t$,
  \begin{itemize}
  \item the \emph{resilience level} $r(t)$ is a value normalized to the
    interval $[0,1]$, with $0$ representing the lowest level of
    resilience (a person who is not resilient at all) and $1$
    representing the maximal level of resilience (a person who is as
    resilient as can be);
  \item the \emph{depression state} $s(t)$ is a value normalized to the
    interval $[0,1]$, with $0$ representing the absence of any
    depression symptoms and $1$ representing depression symptoms at
    their worst.
  \end{itemize}
\end{assumption}

Both of these normalizations merely serve to simplify the presentation. For
our purposes it is not essential how exactly these values are
interpreted, e.g., what ``$70\%$ depressed'' would actually
mean. What is essential is only that both are nonnegative quantities with a
continuous range of values.  There is no inherent restriction to this
approach since commonly used scales such as BDI-II or BHS naturally
have a bounded range of possible values that could easily be
normalized to the interval $[0,1]$ and be interpolated to a continuous
range. Or, vice versa, one could transform the model to a range
compatible with a measurement scale for depression and resilience.

Since depression is commonly associated with adverse events affecting
a person---and not with positive ones---we further assume the following.
\begin{assumption}
  There is an \emph{external stimulus} or \emph{influence} to the model
  denoted by $e=e(t)\geq 0$, with the convention that $e=0$ means
  absence of the stimulus and larger values mean more severe and
  pronounced forms of negative life events affecting the
  individual. This stimulus may change with time $t$.
\end{assumption}
We note that we do not assume that $e$ has a maximal possible
value. While such an assumption could be imposed without harm, it has
no impact on our model. We also note that our model does not aim at
capturing any forms of treatment or positive stimuli.

With this notation the evolution of the symptom is modeled by the
differential equation
  \begin{equation*}
    \frac{d}{dt} s(t) = \Big( e(t) \big(1+s(t)-r(t)\big) - s(t) r(t)   \Big) \big(1-s(t)\big) s(t),
  \end{equation*}
that describes the change rate (time derivative) of $s(t)$ in terms of
the current symptom state $s(t)$, the current resilience level $r(t)$
and the current external stimulus $e(t)$.

Before we discuss the model in more detail, we rewrite it in a
condensed form that is more common in the mathematical and
engineering sciences. In this notation the ``$(t)$'' is dropped and
$\frac{d}{dt} s(t)$ is replaced by the symbol $\dot s$, in order to
have the less cluttered representation
  \begin{equation}
    \label{eq:1-reduced}
    \dot s = \big( \underbrace{e (1+s-r)}_{\text{+}}  \underbrace{-\phantom{(}sr}_{\text{-}} \big) \underbrace{(1-s)~ s}_{\text{confining}}.
  \end{equation}
The second and third factors, $(1-s)$ and $s$, are simply there to
confine the solutions to the interval $[0,1]$. The growth term
$e (1+s-r)$ is always nonnegative. It models that the depression
symptom increases if there is an external influence (non-zero $e$),
and that this rate of change increases with the current symptom
severity $s$ and decreases with the resilience level $r$.  The
decay term $- s r$ is always nonpositive and models the attenuation
of the symptom. The more resilient a person is, the larger in
magnitude this attenuation (read: recovery) is.  The net change rate is
the combination of the growth and decay terms. Up to this point, the
model could be used with a static resilience level $r$, but it will be
more compelling to consider a dynamic resilience level, which we model
by
  \begin{equation}\label{eq:2}
    \dot r = \big((1-s)r - s \big) (1-r) r.
  \end{equation}
Again the second and third factors, $(1-r) r$, ensure that the
resilience remains confined to the interval $[0,1]$. The decay term in
the first factor is simply $-s$, the symptom severity, while the growth
term is $(1-s) r$, modeling that the increase of resilience is
proportional to how good a person feels, i.e., to $(1-s)$, and to their
current resilience $r$.

To write these coupled equations \eqref{eq:1-reduced}--\eqref{eq:2} in
even more compact form, we collect the states $r$ and $s$ into the tuple
$x \coloneqq (r,s)^T$ and obtain the following coupled system of
nonlinear ordinary differential equations
  \begin{equation}
    \begin{split}
      \dot r & = \big( (1-s)r - s \big) (1-r) r\\
      \dot s & =  \big( e (1+s-r) - s r   \big) (1-s) s ,
    \end{split}
    \rlap{\strut\hskip 3.5cm $\biggr\}$}
  \end{equation}
which can equivalently be stated as
  \begin{equation}
    \label{eq:3-RS}
    \dot x
    =  f(x, e) \coloneqq
    \begin{pmatrix}
      \big( (1-s)r - s \big) (1-r) r\\
      \big( e (1+s-r) - s r   \big) (1-s) s
    \end{pmatrix}.
  \end{equation}
Expanding the right-hand side $f$ shows that the
nonlinearity is given in terms of polynomials of total degree four
(and of degree at most three in each of the variables $r$ and $s$), which
allows the model to accommodate a rich set of dynamics as we shall
discuss in the sequel.

Specifically, we show that the model can reproduce the chronic,
delayed, recovery, and resilience scenarios depicted in
Fig~\ref{fig:scenarien}, and, in addition, it captures improvement
from pre-existing conditions as well as burnout from low-grade
adversity, and it can reproduce isolated and recurrent depressive
episodes.

We note that our model differs from other existing mathematical models
for depression.  For instance, the authors of
\cite{demiccheng2014-modeling-the-dynamics-of-disease-states-in-depression}
provide a one-dimensional model for the evolution of the symptom of a
depression, with the external input entering as a Gaussian noise term. Their
model is a polynomial of order three. Due to the single dimension, the
range of the possible dynamics is restricted to movement along a line.

The authors of
\cite{chengdorsognachou2020-mathematical-modeling-of-depressive-disorders:-circadian-driving-bistability-and-dynamical-transitions}
discuss coupled systems of differential equations that build upon HPA
axis models\footnote{HPA axis models are theoretical and computational
  frameworks used to understand how the hypothalamic-pituitary-adrenal
  (HPA) axis functions in both normal and pathological states,
  particularly in relation to stress and depression. HPA axis models
  attempt to explain the complex feedback mechanisms between
  the hypothalamus (which releases corticotropin-releasing hormone),
  the pituitary gland (which releases ACTH, the adrenocorticotropic
  hormone), and the adrenal cortex (which releases cortisol and other
  glucocorticoids).}, along which the effects of the circadian cycle
on input-driven transitions of the HPA axis are analyzed.  In
\cite{karin-et-al2020-new-hpa-model-dysregulation} a mathematical
model consisting of five coupled ODEs is presented, showing that HPA
axis dysregulation after prolonged stress is caused by hormones acting
as growth factors that increase gland masses over weeks, and that
strong glucocorticoid receptor feedback provides resilience by
limiting these mass changes and reducing subsequent hormone
dysregulation.  A comprehensive mathematical study with an even larger
number of states (i.e., number of coupled ordinary differential
equations) has been carried out in
\cite{bothhoogendoornkleintreur2008-modeling-the-dynamics-of-mood-and-depression},
while models based on partial differential equations (which could be
considered a system of infinitely many coupled ODEs) have been
studied in
\cite{tuckwellmiura1978-a-mathematical-model-for-spreading-cortical-depression}.
Other approaches such as in
\cite{kalischbakeral.2017-the-resilience-framework-as-a-strategy-to-combat-stress-related-disorders,
  schaferkunzlerkalischtuscherlieb2022-trajectories-of-resilience-and-mental-distress-to-global-major-disruptions}
consider multiple constituents of resilience, resulting also in more
complex models.

In comparison to our model, it is worth noting that these biochemical
approaches do not take resilience as a state variable into
account. In contrast, our model has just two states, \emph{resilience} and
\emph{depression symptom}, and it deals with the interplay of the states qualitatively. As a
result, our model allows for non-trivial dynamics to reproduce the
typical evolutions of depression described in the literature,
cf.\ Fig~\ref{fig:scenarien}. Moreover, we note that our model is
purely deterministic.

This manuscript is organized as follows. The next section provides a
detailed analysis of stability and monotonicity properties of the
resilience--symptom model. The subsequent section motivates an
alternative interpretation of resilience via memory of past
depression. This is followed by an extensive discussion of characteristic
courses of depression that qualitatively reproduce various scenarios
from the literature and match the authors' own expectations grounded in
anecdotal evidence. Following this, we provide resources to both
reproduce the numerical examples and---with no mathematical or
programming skills required---perform model-based simulations in the web
browser. Finally, we provide a brief summary and an outlook for
possible further developments.

\section*{Analysis and properties of the model}

In this section we provide a detailed mathematical analysis of the
resilience--symptom model~\eqref{eq:3-RS}. Specifically, we compute all
relevant equilibrium points, assess their local stability properties
and regions of attraction, and discuss monotonicity properties of the
model.

\subsection*{Monotonicity}

We start with the investigation of monotonicity properties, since
monotonicity in turn is helpful in establishing stability
properties. Mathematically speaking, monotonicity refers to preservation of order. We shall show that the resilience--symptom model has an
inherent structure, which allows us to compare states of the
model. More precisely, it is reasonable to say that a person with
state $(r_1,s_1)$ is \emph{doing no worse than} a person with state
$(r_2,s_2)$, if both $r_1 \geq r_2$ and $s_1 \leq s_2$. In this case
we will write $(r_1,s_1) \succeq (r_2,s_2)$. In other words, a person
does no worse than another person, if the first person currently has
at least the same level of resilience and at most  as pronounced a
symptom of depression.

Mathematically, the relation $\succeq$ defines a partial ordering,
which can also be stated in terms of the \emph{southeast orthant}
  $$
  K\coloneqq
  \left\{
    \begin{pmatrix}
      x_{1}\\ x_{2}
    \end{pmatrix}
    \in \R^{2} \colon
    x_{1} \geq 0 \text{ and } x_{2}\leq 0
  \right\}.
  $$
That is, in this notation we have 
 \begin{align}\label{eq:def-state-ordering}
  \begin{pmatrix}
    r_{1}\\
    s_{1}
  \end{pmatrix}
  \succeq
  \begin{pmatrix}
    r_{2}\\
    s_{2}
  \end{pmatrix}
  \quad :\!\Longleftrightarrow \quad
  \begin{pmatrix}
    r_{1}\\
    s_{1}
  \end{pmatrix}
  -
  \begin{pmatrix}
    r_{2}\\
    s_{2}
  \end{pmatrix}
  = 
  \begin{pmatrix}
    r_{1}-r_{2}\\
    s_{1}-s_{2}
  \end{pmatrix}
  \in K. 
 \end{align}
The order symbol $\preceq$ is introduced in the obvious way, while
$x\succ y$ is always defined by the simultaneous conditions
$x\succeq y$ and $x\ne y$, and the relation $x\ggcurly y$ is defined
by $x-y\in \interior K$, the interior of $K$. In other words,
$$
  \begin{pmatrix}
    r_{1}\\
    s_{1}
  \end{pmatrix}
  \ggcurly
  \begin{pmatrix}
    r_{2}\\
    s_{2}
  \end{pmatrix}
  \quad \iff \quad
  \begin{pmatrix}
    r_{1}\\
    s_{1}
  \end{pmatrix}
  -
  \begin{pmatrix}
    r_{2}\\
    s_{2}
  \end{pmatrix}
  = 
  \begin{pmatrix}
    r_{1}-r_{2}\\
    s_{1}-s_{2}
  \end{pmatrix}
  \in \interior K 
  \quad \iff \quad
  [r_{1}>r_{2} \quad\&\quad s_{1}<s_{2}].
$$

Since the model under consideration is two-dimensional, we limit the
presentation to $\R^2$, although the theory of monotone systems is far
more developed, cf.~\cite{smith1995-monotone-dynamical-systems}. Next,
we briefly outline what monotone systems are, while restricting the
presentation to the setting of this paper. For the general case, we
refer to \cite{angelisontag2003-monotone-control-systems}.

The function $f$ that defines the first-order ordinary differential
equation in \eqref{eq:3-RS} has continuous partial
derivatives\footnote{In fact, the function $f$ has a natural domain of
  definition that contains an open set containing the domain
  $[0,1]^{2}$, so we may consider derivatives on the boundary of
  $[0,1]^{2}$.} with respect to all variables $x=(r,s)^{T}$ and $e$. Therefore,
for each initial condition
  $$
  \begin{pmatrix}
    r(0)\\ s(0)
  \end{pmatrix}
  =
  \begin{pmatrix}
    r_{0}\\ s_{0}
  \end{pmatrix}
  \in [0,1]^{2}
  $$
and locally Lebesgue integrable input function
$e\colon [0,\infty) \to \R_{+}$, there exists a unique solution,
  $$
  \varphi\colon
  [0,\infty)\to[0,1]^{2},\quad
  t\mapsto
  \begin{pmatrix}
    r(t)\\
    s(t)
  \end{pmatrix}
  $$
which we denote by $\varphi(t) = \varphi\big(t,(r_0,s_0,e)\big)$, and which satisfies
  \begin{align*}
    \varphi\big(0,(r_0,s_0,e)\big) &= (r_{0},s_{0})^{T}\\
    \intertext{and, for all $t\geq 0$,}
    \frac{d}{dt}\varphi\big(t,(r_0,s_0,e)\big) &= f\Big(\varphi\big(t,(r_0,s_0,e)\big), e(t)\Big).
  \end{align*}

To define a partial order for the inputs
$e_{1}, e_{2}\colon [0,\infty) \to \R_{+}$, we say that an external
stimulus $e_{1}$ is \emph{not worse} (for a patient) than stimulus
$e_{2}$, if at every time $t$ it takes at most the same value. In
formulas, we write\footnote{To be more general, one could replace
  the ``for all $t$'' in the order relation definitions for the inputs by a
  ``for almost all $t$'' in the sense of Lebesgue measure.}
 \begin{align}\label{eq:def-input-ordering}
  e_1 \succeq e_2 \qquad :\!\Longleftrightarrow\qquad \Big[ e_1(t) - e_2(t) \leq 0  \text{ for all } t\geq 0 \Big].
 \end{align}
Again, the corresponding induced partial order symbols are defined as
before, and we note that there is no confusion arising from using the
same symbol for the order relations on the input and the state spaces,
as it is always clear from the context which order relation is meant.
The cone corresponding to the partial order on the inputs is thus
given by
  $$
  C = \{ e\colon[0,\infty) \to \R\colon e(t)\leq 0\text{ for all }
  t\geq 0 \}.
  $$
With regard to these partial order relations, system~\eqref{eq:3-RS} is said to be
\emph{monotone}, if for all
  $$
  \begin{pmatrix}
    r_{1}\\
    s_{1}
  \end{pmatrix}
  \succeq
  \begin{pmatrix}
    r_{2}\\
    s_{2}
  \end{pmatrix}
  $$
and $e_1 \succeq e_2$ one has
  $$
  \varphi\big(t,(r_1,s_1,e_1)\big) \succeq
  \varphi\big(t,(r_2,s_2,e_2)\big) \quad \text{ for all } t \geq 0.
  $$
Since $f$ in~\eqref{eq:3-RS} is continuously differentiable,
monotonicity of~\eqref{eq:3-RS} can be verified by the extended
\emph{Kamke condition},
cf.~\cite[Cor.~III.3]{angelisontag2003-monotone-control-systems},
which is as follows.  System~\eqref{eq:3-RS} is monotone if and only if
  \begin{align*}
     (-1) \frac{\partial f_1}{\partial s} (r,s,e) \geq 0
     \quad \text{ and } \quad
     (-1) \frac{\partial f_2}{\partial r} (r,s,e) \geq 0 
  \end{align*}
for all $(r,s) \in [0,1] \times [0,1]$ and for all $ e \geq 0$, and 
  \begin{align*}
     (-1) \frac{\partial f_1}{\partial e} (r,s,e) \geq 0
     \quad \text{ and } \quad
     \frac{\partial f_2}{\partial e} (r,s,e) \geq 0         
 \end{align*}
for all $(r,s) \in [0,1] \times [0,1]$ and for all $ e \geq 0$. 

Indeed, by direct computation, we obtain
  \begin{align*}
     (-1) \frac{\partial f_1}{\partial s} (r,s,e) = r  (1-r) (1+r) \geq 0 
     \quad \text{ and } \quad
       (-1) \frac{\partial f_2}{\partial r} (r,s,e) =s (1-s) \big(e+ s \big) \geq 0      
  \end{align*}
for all $(r,s) \in [0,1] \times [0,1]$ and for all $ e \geq 0$.  Moreover, the inequalities are strict in the interior of the unit square. Furthermore, one has
  \begin{align*}
    (-1) \frac{\partial f_1}{\partial e} (r,s,e) = 0
     \quad \text{ and } \quad
    \frac{\partial f_2}{\partial e} (r,s,e) = s(1-s) \, (1+s-r)\geq 0
     \end{align*}
for all $(r,s) \in [0,1] \times [0,1]$ and for all $ e \geq 0$.  For
future reference, we summarize the above considerations in the
following statement.

\begin{theorem}
  \label{thm:monotonicity-RS}
  Model~\eqref{eq:3-RS} is monotone with respect to the partial orderings
  \eqref{eq:def-state-ordering} and \eqref{eq:def-input-ordering}.
\end{theorem}

\subsection*{Equilibria and stability}

Instead of recalling the rigorous definitions of stability and
attractivity of an equilibrium (in the sense of Lyapunov), we just
give a loose description. For a comprehensive introduction to
stability theory for differential equations we refer the interested
reader to \cite[Ch.~4]{khalil2002-nonlinear-systems} and
\cite[Ch.~3]{hinrichsenpritchard2005-mathematical-systems-theoryi---modelling-state-space-analysis-stability-and-robustness}. Generally
speaking, an equilibrium is \emph{stable} if trajectories starting
nearby remain nearby for all future time.  An equilibrium is
\emph{attractive} if trajectories that start sufficiently close to
the equilibrium converge to the equilibrium as time goes to
infinity. An equilibrium is \emph{asymptotically stable} if it is
stable and attractive.

We begin our stability analysis by stating the equilibria of the
resilience--symptom model.  To this end we consider only states
$(r,s) \in [0,1]\times [0,1]$. Treating the external input
$e\geq 0$ as a fixed parameter, we find the following list of
equilibrium points for model~\eqref{eq:3-RS}
  \begin{gather*}
     p_1 = \binom{0}{0},\quad
     p_2 = \binom{1}{0},\quad
     p_3 = \binom{1}{1},\quad
     p_4 = \binom{0}{1}, \\
     p_5(e)=
     \begin{pmatrix}
       \frac{e + \sqrt{5 \, e^{2} + 4 \, e}}{2 \, {\left(e + 1\right)}}\\
       \frac{\sqrt{5 \, e^{2} + 4 \, e} - e}{2 \, {\left(e + 1\right)}}
     \end{pmatrix} \text{ for every } e\in[0,1],       
     \\
     P_{6} = \left\{\binom{0}{s}\colon s\in[0,1]\right\} \text{ only for }e=0, \rlap{\text{ and}}\\
     P_{7} = \left\{\binom{1}{s}\colon s\in[0,1]\right\} \text{ only for }e=1,
   \end{gather*}
where $p_5$ is within the unit square if and only if $e\in[0,1]$.
We also note that $p_5(0)=p_1$ and $p_5(1)=\binom{1}{\frac{1}{2}}$, in
other words, $p_{5}$ describes a continuous arc of singleton
equilibria, parametrized by the input, while for the specific values
$e=0$ and $e=1$, respectively, the sets $P_{6}$ and $P_{7}$ are
continua of equilibrium points.

The equilibria $p_1$, $p_2$, $p_3$, and $p_4$ are the corners of the unit
square. We will refer to $p_2$ as \emph{bliss}, and $p_4$ as
\emph{despair}, cf.\ Fig~\ref{fig:equilibria}.

\begin{figure}[htb]
  \centering
  \includegraphics[width=0.4\linewidth]{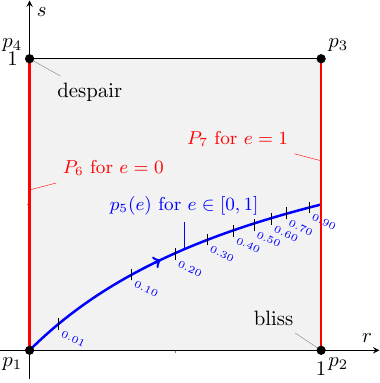}
  \caption{%
    \label{fig:equilibria}%
    \textbf{Equilibria of system~\eqref{eq:3-RS}.}  The obvious
    equilibria of model~\eqref{eq:3-RS} are in the corners of the
    state space. Along the curve shown in blue (from left to right
    with increasing $e$) are equilibrium points for fixed $e\in[0,1]$.
    The red boundaries (left and right) of the unit square are
    equilibrium sets (continua) for $e=0$ and $e=1$, respectively.
    \emph{Bliss} refers to the hypothetical state of perfect
    resilience and no depression symptoms, while \emph{despair} refers
    to the situation of `complete' depression and no resilience.}
\end{figure}

It is an artifact of the design choices in this model that the points
$p_1, p_2, p_3,$ and $p_4$ are equilibria irrespective of the value
of $e$, meaning that an individual who has, say, zero symptoms and
resilience will always remain this way, unfazed by any external
stimuli. For practical purposes, however, it is more realistic
to consider whether the states asymptotically converge to
zero, or remain close to zero, which is a question of (asymptotic)
stability. Put differently, no living individual is
likely ever to be found in states $p_{1}$, $p_{2}$, $p_{3}$, or $p_{4}$, as
everyone will have had some negative experiences in their
lives.

\begin{figure}[htb]
  \centering
  \includegraphics[width=0.4\linewidth]{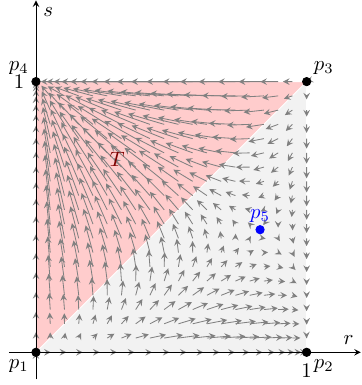}
  \caption{%
    \label{fig:quiver}%
    \textbf{A vector field (quiver) plot for $f$ in \eqref{eq:3-RS}
      with a constant value of $e=0.6$.} Note that for a fixed value of
    the input $e\in(0,1)$ there does not exist a continuum of
    equilibria, as the previous figure may suggest: The equilibria in
    $[0,1]^{2}$ are all singletons. The region shaded in red is the
    invariant region $T$ from
    Theorem~\ref{thm:stability}\ref{item:RS-upper-triangle-invariance-and-ROA}.
  }
\end{figure}

Before we turn to stability considerations, we have to clarify why it
is permissible to restrict the state-space of system~\eqref{eq:3-RS}
to $[0,1]^{2}$, which we will do implicitly henceforth, that is, we
will always consider stability relative to this state-space only.

\begin{lemma}
  \label{lemma:invariance-of-unit-square}
  The unit square $[0,1]^{2}$ is positively invariant
  under~\eqref{eq:3-RS}. That is, trajectories starting in $[0,1]^{2}$
  remain in this set for all future time. 
\end{lemma}

\begin{proof}
  The conclusion is a consequence of monotonicity: For all
  $(r_{0},s_{0})\in[0,1]^{2}$, all $e\colon[0,\infty)\to\R_{+}$,
  and all $t\geq0$ one has
  $$
  \text{\emph{despair}} = 
  \varphi\big(t,(\text{\emph{despair}},e)\big) \preceq
  \varphi\big(t,((r_{0},s_{0}),e)\big) \preceq 
  \varphi\big(t,(\text{\emph{bliss}},e)\big) = \text{\emph{bliss}}
  $$
  by Theorem~\ref{thm:monotonicity-RS}. Thus, the trajectory cannot
  escape $[0,1]^{2}$ and hence exists for all future time.  
\end{proof}

The standard method to tackle stability properties of an equilibrium
of a nonlinear differential equation is to consider the linearization
around it. The linearizations are given by the Jacobian matrices of
$f$ evaluated at the corresponding equilibria. However, this method
cannot always be applied, as we will see in the subsequent stability
analysis. The following result summarizes the stability properties of
the equilibria of model~\eqref{eq:3-RS}.
\begin{theorem}\label{thm:stability}
  For model~\eqref{eq:3-RS} and fixed values of $e$, it holds that
  \begin{enumerate}[label=(\alph*), ref=(\alph*)]
  \item\label{item:RS-stability-p1-p3} the equilibria $p_{1}$ and $p_{3}$ are unstable for
    $e\geq 0$;
  \item\label{item:RS-instability-p5} for $e\in(0,1)$ the equilibrium $p_{5}(e)$ is unstable;
  \item\label{item:RS-stability-bliss} the equilibrium $p_{2}$ \emph{(bliss)} is stable
    if $e=1$; locally asymptotically stable for all
    $e\in [0,1)$; and \emph{unstable} for $e>1$;
  \item\label{item:RS-AS-despair} the equilibrium $p_{4}$ \emph{(despair)} is locally
    asymptotically stable for all $e>0$ and stable for $e=0$;
  \item\label{item:RS-upper-triangle-invariance-and-ROA} the upper
    left triangle $T= \{(r,s)\in[0,1]^{2}\colon s\geq r\}$ is
    positively invariant under \eqref{eq:3-RS} for all $e\geq0$;
    moreover, if $e>0$, then $T\setminus\{p_{1},p_{3}\}$ is contained
    in the domain of attraction for $p_{4}$ (despair).
  \end{enumerate}
\end{theorem}

The interpretation of statement~\ref{item:RS-stability-bliss} is that
even though a patient may be affected by not-too-intense negative
external stimuli, as long as they are sufficiently close to bliss from
the beginning, they will remain close to bliss and even get closer to
bliss as time progresses.

Likewise, if someone is too close to despair already at the beginning,
they will remain close and approach despair over time (as long as the
external stimulus is not exactly zero). In other words, once a patient
is too close to despair, there is no recovery\footnote{We should
  clarify that there is ``no recovery'' in this model, which is only a
  mathematical model not describing any real person and, in
  particular, this model does not take into account any form of
  intervention or treatment at all.}. In fact, there is a universal
threshold: once the symptoms are at least as strong as the resilience
and the external stimulus is not identically zero, a patient cannot
recover from the depression and will approach despair.
\begin{proof}
  \begin{description}
  \item [We start with part~\ref{item:RS-stability-p1-p3}.]% 
    We
    consider the corresponding linearizations. The
    Jacobian matrices
      \begin{align*}
        J_f(p_{1}) =  \begin{pmatrix} 0 & 0 \\ 0 &e \end{pmatrix} \quad \text{and} \quad
        J_f(p_{3}) =  \begin{pmatrix} 1  & 0 \\ 0& 1-e  \end{pmatrix}
      \end{align*}
    have the positive eigenvalues $e$ and $1$, respectively, so that for
    $e>0$ the assertion follows immediately
    from~\cite[Thm.~4.7]{khalil2002-nonlinear-systems}. For $e=0$ the
    matrix $J_{f}(p_{3})$ still has the positive eigenvalue $1$ and
    $p_{3}$ is unstable as well. The matrix $J_f(p_{1})$, however, is
    zero for $e=0$, so that the linearization method cannot be
    applied. To show
    instability in this case we consider the boundary section
      \begin{align*}
        M = \left\{  \binom{r}{s} \in [0,1]^2 \colon s=0 \right\}
      \end{align*}
    of the unit square. The set $M$ is invariant: Indeed, from
    \eqref{eq:3-RS} it follows that $\dot s =0$ for every $(r,s) \in
    M$. Moreover, it holds that
      \begin{align*}
        \dot r = (1-r)r^2 >0 \qquad \text{ for all }(r,s) \in M \text{ with } r\notin\{0,1\}.
      \end{align*}
    Thus, from every neighborhood of the origin emerges a trajectory
    with initial condition $(r_0,s_0) \in M$ that evolves away from
    the origin along $M$, showing instability of $p_{1}$.  This
    completes the proof of part~\ref{item:RS-stability-p1-p3}.

  \item[For part~\ref{item:RS-instability-p5}] we simply note that the
    Jacobian matrix $J_{f}(p_{5})$ has one positive and one negative
    eigenvalue when $e\in(0,1)$, cf.\ Fig~\ref{fig:eigenvalues-Jfp5},
    and the argument is thus the same as for
    part~\ref{item:RS-stability-p1-p3}. More precisely, with
    $D\coloneqq\sqrt{5e^{2}+4e}$ one has
    $\det J_{f}(p_{5}) = \frac{e\,D\,(D-2e-1)}{(e+1)^{2}}$, and
    $D<2e+1$ if and only if $5e^{2}+4e < 4e^{2}+4e+1$ if and only if
    $e<1$. Hence $\det J_{f}(p_{5})<0$ for $e\in(0,1)$, i.e.\
    $p_{5}(e)$ is a saddle.
    \begin{figure}[htb]
      \centering
      \includegraphics[width=0.6\linewidth]{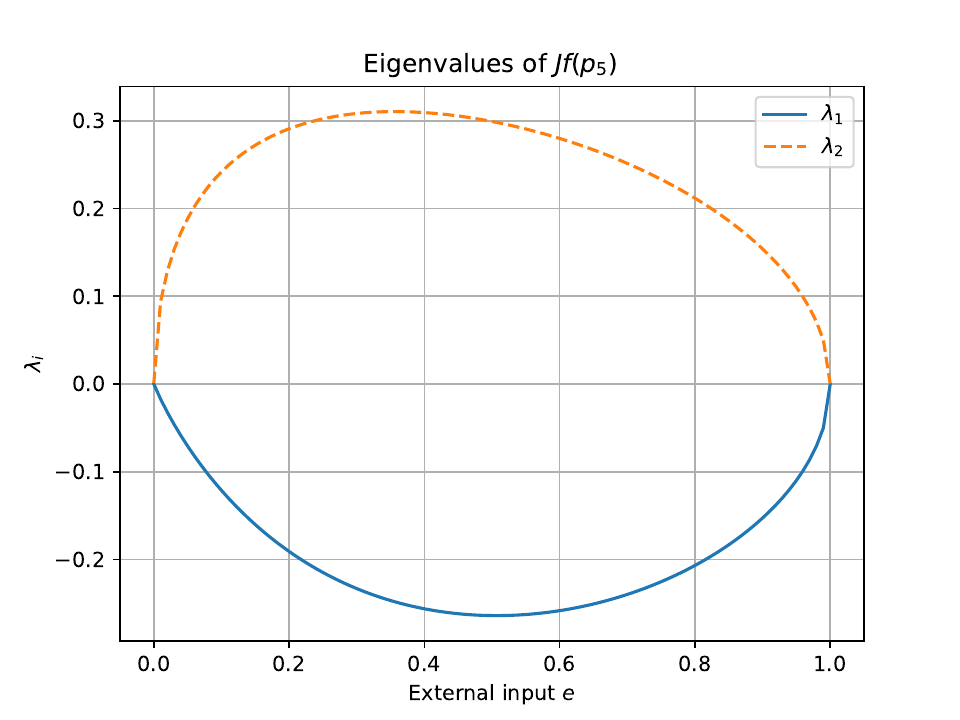}
      \caption{%
        \label{fig:eigenvalues-Jfp5}%
        \textbf{The eigenvalues of $J_f(p_5)$ as a function of $e$.}}
    \end{figure}
    
  \item[For part~\ref{item:RS-stability-bliss}] we consider the
    function
      \begin{align*}
        V(r,s)\coloneqq (1-r) +s,
      \end{align*}
    which is continuous in $(r,s)$, zero only at $p_{2}$ (bliss), and
    positive elsewhere in the unit square. The derivative of $V$ along
    solutions to \eqref{eq:3-RS} is given by
      \begin{align*}
        \dot V(r,s)&=\tfrac{\operatorname{d}}{\operatorname{d}t}V(x(t)) =\nabla V(r,s)^T\cdot f(r,s,e) \\
                   &=\frac{\partial V}{\partial r}(r,s) \cdot \big( (1-s)r - s \big) (1-r) r
                     + \frac{\partial V}{\partial s}(r,s) \cdot \big( e (1+s-r) - s r   \big) (1-s) s\\
                   &= - \big( (1-s)r - s \big) (1-r) r
                     + \big( e (1+s-r) - s r   \big) (1-s) s.
      \end{align*}
    In a first step we show that there is a $\delta \in (0,1]$ such
    that
      \begin{align*}
        \dot V(r,s)\leq 0
      \end{align*}
    for all
      \begin{align*}
        (r,s) \in \Omega_\delta \coloneqq \{ (r,s) \in [0,1] \times [0,1] \colon  1-r+s < \delta \} 
      \end{align*}
    and for all $e\in [0,1]$.  Then $V$ defines a Lyapunov function on
    $\Omega_\delta$ and the stability of \emph{bliss} follows from
    \cite[Thm.~4.1]{khalil2002-nonlinear-systems}.
    
    To verify this, we first note that
      \begin{align}\label{eq:V-dot-bliss-r=1}
        \dot V (1,s)=  \, (1-s)s^2(e-1) \leq 0 
      \end{align}
    for all $s \in [0,1]$ and for all $e \in [0,1]$.  To treat the case $0<r<1$,  we rewrite $\dot V (r,s)$ as follows 
      \begin{align*}
        \dot V (r,s)
        =&-\, (1-r)r^2  +  \left( (1-r^2)r + e(1-r) \right)\, s  +   r(e-1) \, s^2 - (e -r) \, s^3 .
      \end{align*}
    Considering $r \in (0,1)$ as a fixed but arbitrary parameter, we
    see that $\dot V(r,s)$ defines a cubic polynomial in $s$, say
    $p_r(s)$. To make this more explicit, we will write
      \begin{align*}
        \dot V (r,s) = p_r(s)=  a_0(r) + a_1(r) s + a_2(r) s^2 +a_3(r) s^3
      \end{align*}
    with the $r$-dependent coefficients
      \begin{align*}
        a_0(r) &\coloneqq -\, (1-r)r^2, \\
        a_1(r) &\coloneqq (1-r^2)r + e(1-r) , \\
        a_2(r) &\coloneqq r(e-1) , \\
        a_3(r) &\coloneqq r-e.
      \end{align*}
    For all $r\in(0,1)$ and all $e \in [0,1]$, the coefficients satisfy $a_0(r) = -(1-r)r^{2} < 0$ and
      \begin{align*}
        a_1(r) = (1-r)\big( r(1+r) + e \big) \leq 3\,(1-r).
      \end{align*}
    Moreover, since $r(e-1) + (r-e)s \leq 1-r$ for all $r,s,e\in[0,1]$
    (indeed, this is equivalent to $re + (r-e)s \leq 1$, whose left-hand
    side is linear in $s$ with the endpoint values $re \leq 1$ at $s=0$
    and $re + r - e = r + e(r-1) \leq 1$ at $s=1$), the two highest
    order terms admit the bound
      \begin{align*}
        a_2(r)\, s^2 + a_3(r)\, s^3 = \big( r(e-1) + (r-e)\, s \big)\, s^2 \leq (1-r)\, s^2.
      \end{align*}
    Now set $\delta \coloneqq \tfrac16$. Every $(r,s) \in \Omega_\delta$ with $r < 1$ satisfies $r > 1-\delta$ and $s < \delta$, hence
      \begin{align}\label{eq:V-dot-bliss-r-in-0-1}
        \dot V (r,s) = p_r(s) \leq (1-r) \big( {-r^{2}} + 3 s + s^{2} \big)
        < (1-r) \big( {-(1-\delta)^{2}} + 3\delta + \delta^{2} \big)
        = (1-r) ( 5\delta - 1 ) < 0
      \end{align}
    for all $e \in [0,1]$.  Consequently, for this choice of $\delta$,
    and taking \eqref{eq:V-dot-bliss-r=1} into account for the
    boundary case $r=1$, we have that $V$ defines a Lyapunov function
    on the neighborhood $\Omega_\delta$ of bliss. Hence, bliss is a
    stable equilibrium.

    To treat the case $e<1$, we reinspect the above inequalities \eqref{eq:V-dot-bliss-r=1} and \eqref{eq:V-dot-bliss-r-in-0-1}. It turns out that 
      \begin{align*}
        \dot V(r,s) <0 \qquad \text{ for all } (r,s) \in \Omega_\delta  \setminus \{(1,0)\}.
      \end{align*}
    Applying \cite[Thm.~4.1]{khalil2002-nonlinear-systems}, we conclude that bliss is asymptotically stable if $e<1$.

    Now consider $e>1$. The edge $r=1$ is invariant, i.e., $r$ remains
    constant but $s$ increases with $\dot s = s^{2}(1-s)(e-1)>0$,
    hence bliss is unstable.

  \item[To see part~\ref{item:RS-AS-despair},] we first treat the case
    $e \neq 0$. In this case, the linearization is given by
      \begin{align*}
        J_f\left(\binom{0}{1} \right) =  \begin{pmatrix} -1&0  \\ 0 &-2e \end{pmatrix}.
      \end{align*}
    As $e \neq 0$, both eigenvalues are negative and using
    \cite[Thm.~4.7]{khalil2002-nonlinear-systems} we conclude that
    despair is locally asymptotically stable for all $e>0$.
    
    For the case $e=0$, we shall show that 
      \begin{align*}
        W(r,s) \coloneqq r + (1-s),
      \end{align*}
    which is continuous in $(r,s)$, zero only at $p_{4}$ (despair), and
    positive elsewhere in the unit square, defines a Lyapunov function
    on some neighborhood of despair. The assertion follows from
    \cite[Thm.~4.1]{khalil2002-nonlinear-systems}.

    The derivative of $W$ along solutions of \eqref{eq:3-RS} for $e=0$ is given by
      \begin{align*}
        \dot W(r,s)& =\nabla W(r,s)^T\cdot f(r,s,0)
                     =  \big( (1-s)r - s \big) (1-r) r +  r   (1-s) s^2\\
                   &= r \Big( \big( (1-s)r - s \big) (1-r) +  (1-s) s^2 \Big).
      \end{align*}
    To specify the neighborhood, set $\delta\coloneqq\tfrac13$. For all
      \begin{align*}
        (r,s) \in U_\delta \coloneqq \{ (r,s) \in [0,1] \times [0,1] \colon  r < \delta, \, s > 1-\delta \}
      \end{align*}
    we get that
      \begin{align*}
        (1-s)r - s   \leq (1-s) \delta - s
        = \delta - s (1+\delta)
        <  \delta - (1-\delta) (1+\delta)
        = \delta^2 + \delta -1 <0.
      \end{align*}
    Using $1-r > 1-\delta$ and $0 \leq (1-s)s^2 \leq 1-s < \delta$, this yields
      \begin{align*}
        \big( (1-s)r - s \big) (1-r) +  (1-s) s^2
        < (\delta^2+\delta-1)(1-\delta) + \delta = -\delta^{3} + 3\delta - 1 = -\tfrac{1}{27} < 0.
      \end{align*}
    Hence, $\dot W (r,s) \leq 0 $ for all $(r,s) \in U_\delta$, with
    equality if and only if $r=0$. That is, $W$ defines a Lyapunov
    function on the neighborhood $U_\delta$ and by
    \cite[Thm.~4.1]{khalil2002-nonlinear-systems} despair is a stable
    equilibrium for $e=0$.

  \item[Now to part~\ref{item:RS-upper-triangle-invariance-and-ROA}.]%
    To show the invariance of the upper left triangle $s\geq r$, we
    consider the Lyapunov function $W$. Along the line $s=r$ we find
    that
      \begin{align*}
        \dot W(s,s) = {\left(s^{2} - s\right)} e \leq 0
      \end{align*}
    whenever $s\in[0,1]$ and $e\geq 0$. According to
    \cite[Thm.~16.9]{Aman-1990-ODE-engl} this shows that the triangle
    $s\geq r$ is invariant under~\eqref{eq:3-RS}.

    That all trajectories starting in $T\setminus\{p_{1},p_{3}\}$
    converge to despair can be seen in Fig~\ref{fig:quiver}.  A formal
    proof is as follows. Recall that, by the Poincar\'{e}--Bendixson
    Theorem, cf.\
    \cite[Thm.~3.1.14]{hinrichsenpritchard2005-mathematical-systems-theoryi---modelling-state-space-analysis-stability-and-robustness},
    together with \cite[Thm.~3.22]{HirschSmith2005-handbook},
    Theorem~\ref{thm:monotonicity-RS} yields that the $\omega$-limit
    sets can only consist of equilibria. Hence, it remains to verify
    that no trajectory starting in $T$ converges to $p_1$ or $p_3$.

    We start with $p_1$. Let $\tau>0$. Then for all $0<r<\tau$ and
    $1>s>0$ we get that
      \begin{align*}
        \dot s &= ( e(1+s-r) -sr )(1-s)s \\
               &>( e(1+s-\tau) -s\tau )(1-s)s = e(1-s^2)s - \tau \,s(1-s)(e+s).
      \end{align*}
    Choosing $\tau$ sufficiently small, we obtain that $\dot s >0$
    around the equilibrium $p_1$. Hence, trajectories in $T$ cannot
    converge to $p_1$.

    To treat the equilibrium $p_3$, let $\gamma >0$. Then for all
    $1>s> 1-\gamma$ and $0<r<1$ it holds that
      \begin{align*}
        \dot r &= ( (1-s) r  -s  )(1-r)r \\
               &<  ( \gamma r  - (1-\gamma)  )(1-r)r  = -r(1-r) + \gamma (1-r^2)r.
      \end{align*}
    Choosing $\gamma$ sufficiently small, we obtain that $\dot r <0$
    around the equilibrium $p_3$. Hence, trajectories in $T$ cannot
    converge to $p_3$.

    For completeness, we mention that
    $p_{2}=(1,0)^{T}\notin{T}$ because $s<r$ there, and
    $p_{5}(e)\notin{T}$ because its two components differ by
    $r-s=\frac{e}{e+1}>0$ for $e>0$. For $e=1$ one additionally has
    $P_{7}\cap{T}=\{p_{3}\}$, which has been discussed
    already.\qedhere
  \end{description}
\end{proof}

\begin{remark}
  The asymptotic stability of \emph{bliss} cannot be concluded
  from~\cite[Thm.~4.7]{khalil2002-nonlinear-systems} as the Jacobian
    \begin{align*}
      J_f  \binom{1}{0} =
      \begin{pmatrix}
        -1 & 0 \\
        0 & 0
      \end{pmatrix}
    \end{align*}
  has the eigenvalue zero. Moreover, bliss cannot be asymptotically
  stable for $e=1$ as $f(1,s,1)=0$ for all $s\in[0,1]$. Hence, no
  points on the boundary $r=1$ of the unit square are attracted by
  \emph{bliss}.
\end{remark}

One significance of Theorem~\ref{thm:stability} is that if an
individual with mild depression symptoms and very high resilience
experiences no further external stimuli, then their state remains
close to bliss and converges to it as time progresses. In particular,
the depression symptom eventually dissipates to zero, while the
resilience level approaches its maximum.

\section*{Memory versus resilience}

Some psychological theories emphasize the role of memory in the course
of depressive illness. One example is `schema therapy' as developed
by J.E.~Young based on ideas of A.T.~Beck's `cognitive
therapy'~\cite{youngs.weishaar2003-schema-therapy:-a-practitioners-guide}. Young's
schema therapy (read: qualitative model) is among the `third wave of
cognitive-behavioral therapies' and centers around the concept of
maladaptive schemata. Schemata or schemas are a heuristic technique to
encode and retrieve memories.

Instead of considering a selection of specific schemas, we concentrate
in our discussion on one feature common to all of them: they are
linked to a biased memory of new (and old) events and stimuli. We
therefore focus simply on the ``memory'' aspect and do not
discriminate between different biases or schema categories. Our aim is
to demonstrate that reasoning similar to that used in models of
resilience can also be employed to develop mathematical models of
schemas and psychological disorders. Other well-established factors
shaping resilience and vulnerability include stress
sensitization/kindling \cite{Segal1996,Post1992,Monroe2005},
allostatic load \cite{McEwenStellar1993,McEwen1998}, hysteresis
\cite{Cramer2016,Borsboom2017}, and network connectivity
\cite{BorsboomCramer2013,Borsboom2017}.

Taken together, these mechanisms suggest that psychological
vulnerability can be understood as an emergent property of a dynamical
system whose current state depends not only on its present inputs, but
also on its history and accumulated load. This perspective provides a natural
bridge to resilience--symptom models, in which the persistence of
maladaptive states and the system's ability to recover from
perturbations can be studied in terms of the same dynamical
principles.

From a dynamical systems perspective, memory and resilience can be
viewed as two opposing aspects of the same system. Memory reflects the
persistence of the influence of past states on current dynamics,
whereas resilience reflects the ability of a system to return to a
previous state following a perturbation. We therefore consider them as
qualitatively opposing properties, without imposing a specific
functional relationship between them. Whether this relationship is
linear, nonlinear, or takes another form is ultimately an empirical
question that would require appropriate data to determine.

In the following, we subsume this persistence of past states in one
variable, which we call \emph{memory of depression} and denote by
$m$. Here, memory does not refer to memory in the cognitive sense, but
to the persistence of the influence of past states on the current
dynamics of the system. As a first-order approximation, we take memory
to be the counterpart of resilience and assume
	\begin{align}
		\label{eq:m=1-r}
		m=1-r.
	\end{align}
This assumption reflects the idea that systems with high resilience
are predominantly driven by current environmental inputs, whereas
systems with low resilience exhibit stronger internally sustained
dynamics due to the persistent influence of their history. We
emphasize that $m=1-r$ is a simplifying modeling assumption rather
than an empirically established relationship. It allows us to
investigate, within the same mathematical framework, how increasing
persistence of past states affects the dynamics and stability of
depressive symptoms.

Consequently, we may consider the resilience--symptom model under the
specific change of variables given by Eq.~\eqref{eq:m=1-r}. In this
first-order approximation, resilience and memory are therefore
represented as two complementary properties of the system. In the
state variables $(m,s)^{T}$ and with the same external influence
$e\colon [0,\infty) \to [0,\infty)$ as before, the basic
\emph{memory--symptom} model arises from model \eqref{eq:3-RS} via the
coordinate change~\eqref{eq:m=1-r}, and it takes the form
	\begin{equation}
		\label{eq:ode-sympt-memory}
		\begin{pmatrix}
			\dot 
			m\\
			\dot 
			s
		\end{pmatrix}
		=g(m,s,e)\coloneqq
		\begin{pmatrix}
			\big(s - (1-s) (1-m) \big) (1-m) m\\
			\big( e( s+m) - s (1-m) \big) (1-s) s
		\end{pmatrix}.
	\end{equation}
The analysis of the resilience--symptom model carries over to the
memory--symptom model with only minor changes, so that we only
state the results.

A patient with state $(m_{1},s_{1})^{T}$ \emph{does no worse than} a
patient with state $(m_{2},s_{2})^{T}$, if their memory and symptom
levels do not exceed those of the other patient. More formally, we define the
partial ordering
	\begin{align}\label{eq:def-cone-states-ms-model}
		\begin{pmatrix}
			m_{1}\\
			s_{1}
		\end{pmatrix}
		\succeq
		\begin{pmatrix}
			m_{2}\\
			s_{2}
		\end{pmatrix}
		\quad :\!\Longleftrightarrow \quad
		\begin{pmatrix}
			m_{1}-m_{2}\\
			s_{1}-s_{2}
		\end{pmatrix}
		\in K_m\coloneqq
		\left\{   \begin{pmatrix}
			x\\
			y
		\end{pmatrix} \in \mathbb{R}^2
		\colon 
		x \leq 0 ,\, y\leq 0
		\right\}.
	\end{align}
That is, the partial ordering is defined by the southwest orthant. The
same analysis as for the resilience--symptom model applies here as
well. To verify the monotonicity of this model, we will again show
that the extended Kamke condition holds. Indeed, it holds that
	\begin{align*}
		\frac{\partial g_1}{\partial s} (m,s,e) =m  (1-m) (2-m) \geq 0 
		\quad \text{ and } \quad
		\frac{\partial g_2}{\partial m} (m,s,e) =  s (1-s) (e+ s )  \geq 0      
	\end{align*}
for all $(m,s) \in [0,1] \times [0,1]$ and for all $ e \geq 0$. As the
partial ordering with respect to the inputs, we choose again
\eqref{eq:def-input-ordering}. The remaining extended Kamke conditions
read
	\begin{align*}
		\frac{\partial g_1}{\partial e} (m,s,e) = 0
		\quad \text{ and } \quad
		\frac{\partial g_2}{\partial e} (m,s,e) = s(1-s) \, (s+m)\geq 0       
	\end{align*}
for all $(m,s) \in [0,1] \times [0,1]$ and for all $ e
\geq 0$. Consequently, we have the following statement, in analogy with
Theorem~\ref{thm:monotonicity-RS} and Lemma~\ref{lemma:invariance-of-unit-square}.

\begin{theorem}
  The memory--symptom model \eqref{eq:ode-sympt-memory} is monotone on
  the unit square with respect to the partial orderings
  \eqref{eq:def-cone-states-ms-model} and
  \eqref{eq:def-input-ordering}. Moreover, the set $[0,1]^{2}$ is
  positively invariant under \eqref{eq:ode-sympt-memory}.
\end{theorem}

The equilibria of the memory--symptom model
\eqref{eq:ode-sympt-memory} are those of the resilience--symptom model
under the change of coordinates~\eqref{eq:m=1-r}, cf.\
Fig~\ref{fig:quiver-ms}. The following list collects the
equilibria within the unit square $[0,1]\times [0,1]$.

	\begin{gather*}
		q_1 = \binom{0}{0},\quad
		q_2 = \binom{1}{0},\quad
		q_3 = \binom{1}{1},\quad
		q_4 = \binom{0}{1},\\
		q_{5}(e) =  \begin{pmatrix}
			\frac{e+2  - \sqrt{5 \, e^{2} + 4 \, e}}{2 \, {\left(e + 1\right)}}\\
			\frac{\sqrt{5 \, e^{2} + 4 \, e} - e}{2 \, {\left(e + 1\right)}}
		\end{pmatrix} \text{ for every } e\in[0,1],\\
		Q_{6} = \left\{\binom{0}{s}\colon s\in[0,1]\right\} \text{ only for } e=1, \rlap{\text{ and}}\\
		Q_{7} = \left\{\binom{1}{s}\colon s\in[0,1]\right\} \text{ only for }e=0.
	\end{gather*}
Following the interpretation of the resilience--symptom model, we call
$q_1$ and $q_3$ bliss and despair, respectively. Also, we note that
$q_5(0)= q_2$ and $q_5(1) = (0,\tfrac{1}{2})^T$.

\begin{figure}[htb]
  \centering
  \includegraphics[width=0.4\linewidth]{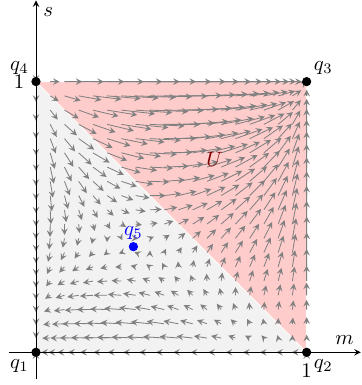}
  \caption{%
    \label{fig:quiver-ms}%
    \textbf{Vector field and equilibria for
    model~\eqref{eq:ode-sympt-memory} when $e=\nicefrac13$.}}
\end{figure}

Adapting the arguments given for the resilience--symptom model to the
present setting---and illustrated by the quiver plot in
Fig~\ref{fig:quiver-ms}---yields the following result.

\begin{theorem}
  The stability properties of the equilibria of the memory--symptom
  model for fixed values of $e$ are as follows.
  \begin{enumerate}[label=(\alph*), ref=(\alph*)]
  \item the equilibria $q_{2}$ and $q_{4}$ are unstable for
    $e\geq 0$;
  \item for $e\in(0,1)$ the equilibrium $q_{5}(e)$ is unstable;
  \item the equilibrium $q_{1}$ \emph{(bliss)} is stable
    if $e=1$; locally asymptotically stable for all
    $e\in [0,1)$; and unstable for $e>1$;
  \item the equilibrium $q_{3}$ \emph{(despair)} is locally
    asymptotically stable for all $e>0$ and stable for $e=0$;
  \item the upper right triangle
    $U= \{ (m,s)\in [0,1]^{2}\colon m+s\geq 1\}$ is positively
    invariant under \eqref{eq:ode-sympt-memory} for all $e\geq0$;
    moreover, if $e\neq 0$, then $U\setminus\{q_{2},q_{4}\}$ is
    contained in the domain of attraction for $q_{3}$
    \emph{(despair)}.
  \end{enumerate}
\end{theorem}

The two locally asymptotically stable equilibria admit a natural
interpretation as two distinct dynamical regimes. The equilibrium
$(m,s)=(0,0)$ represents a resilient, non-depressed state in which the
influence of past depressive states is minimal in its neighborhood and
the system is predominantly responsive to current environmental
inputs. In contrast, the vicinity of $(m,s)=(1,1)$ represents a
persistent depressive regime in which symptoms and the influence of
past depressive states reinforce each other, resulting in a strongly
history-dependent system. The coexistence of these two stable states
implies that the long-term state of the system depends not only on its
current environment, but also on its trajectory and initial condition,
providing a natural interpretation in terms of path
dependence. Moreover, local asymptotic stability implies that
sufficiently small perturbations do not induce a transition between
these regimes, whereas a sufficiently large perturbation is required
to move the system from the healthy regime to the depressive
regime. In this sense, resilience can be interpreted as the ability of
the system to absorb perturbations while remaining within its current
dynamical regime.

\section*{Characteristic courses of depression and model properties}

Now we focus on simulating the model in various scenarios. One
objective is to match the model to figures shown in the
literature. Another aim is to reproduce additional characteristic and
expected behaviors of a depression. Lastly, we showcase some of the
structural properties of the model, and how they translate to courses
of illness.

\subsection*{Reproducing literature trajectory types}

To assess whether the proposed model can reproduce qualitatively
distinct patterns of resilience dynamics described in the literature,
we numerically simulate the model for different parameter
configurations and initial conditions. We compare the resulting
trajectories with the characteristic resilience trajectories reported
in the literature. The aim is not to fit the model to empirical data,
but to examine whether the dynamical mechanisms identified in our
analysis are sufficient to generate qualitatively similar patterns of
resilience.

Figures similar to Fig~\ref{fig:scenarien} and
Fig~\ref{fig:alternative-scenarios} feature prominently in the
resilience literature
\cite{galatzer-levyhuangbonanno2018-trajectories-of-resilience-and-dysfunction-following-potential-trauma:-a-review-and-statistical-evaluation,
  bonanno2004-loss-trauma-and-human-resilience:-have-we-underestimated-the-human-capacity-to-thrive-after-extremely-aversive-events,
  bonannochengalatzer-levy2023-resilience-to-potential-trauma-and-adversity-through-regulatory-flexibility}.
\begin{figure}[htb]
  \centering
  \includegraphics[width=0.8\linewidth]
  {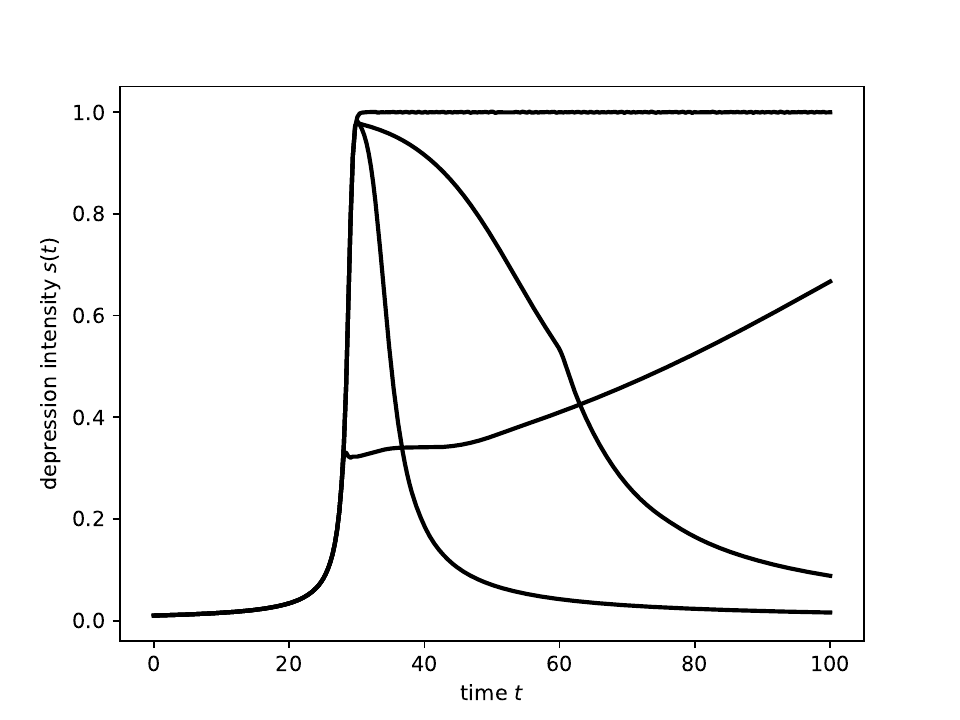}
  \caption{%
    \label{fig:alternative-scenarios}%
    \textbf{Alternative scenarios similar to Fig~\ref{fig:scenarien}
      but based on adverse input with larger values than one and
      matching
      \cite[Fig~1]{galatzer-levyhuangbonanno2018-trajectories-of-resilience-and-dysfunction-following-potential-trauma:-a-review-and-statistical-evaluation}
      more closely.}}
\end{figure}
Here we discuss how these different scenarios are reproduced by
model~\eqref{eq:3-RS}. To this end we will focus on the trajectories
shown in Fig~\ref{fig:scenarien} rather than
Fig~\ref{fig:alternative-scenarios} for simplicity.

\subsubsection*{The chronic scenario}

The \emph{chronic scenario} demonstrates that an initially
non-depressed patient with low resilience can be catapulted into
severe and ongoing depression through a single adverse event. This
would typically happen if the patient is not very resilient to begin
with, cf.\ Fig~\ref{fig:chronic-scenario}.

\begin{figure}[htb]
  \centering
  \includegraphics[width=0.8\linewidth]{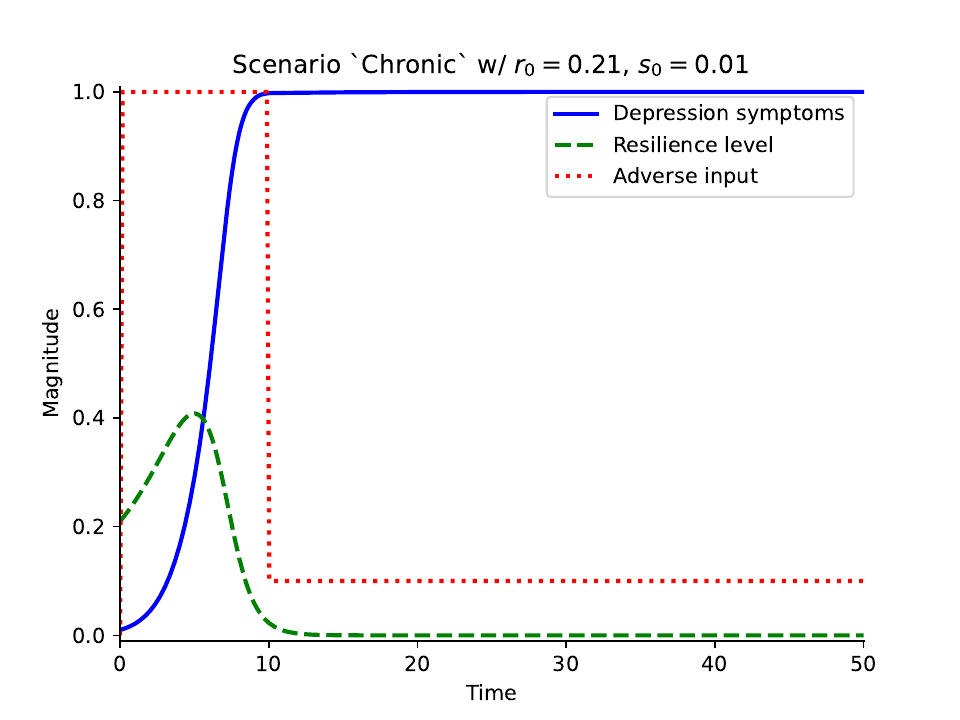}
  \caption{%
    \label{fig:chronic-scenario}%
    \textbf{Chronic scenario.}
    A scenario of a patient becoming chronically depressed
    after an adverse trauma period (time 0 to 10), followed by a
    constant minor adverse input. The patient starts out essentially
    not depressed, albeit with only a very low resilience
    level of about $0.2$. The resilience rises initially while the
    depression symptom is still small; however, the latter rises due
    to the adverse input. Once a certain threshold of depression
    symptom is reached, the resilience level starts to decline and the
    patient is catapulted into a state of severe depression and no
    resilience.}
\end{figure}

\subsubsection*{The delayed scenario}

Following an adverse event, in the \emph{delayed scenario} a patient does
develop depression symptoms more slowly. Here we have chosen an
adverse event that is shorter in duration than in the chronic example,
while keeping the same initial resilience and depression symptom
levels. Yet, due to the persistent background adversity level, the
patient eventually meets the same fate as in the previous example,
cf.\ Fig~\ref{fig:delayed-scenario}.

\begin{figure}[htb]
  \centering
  \includegraphics[width=0.8\linewidth]{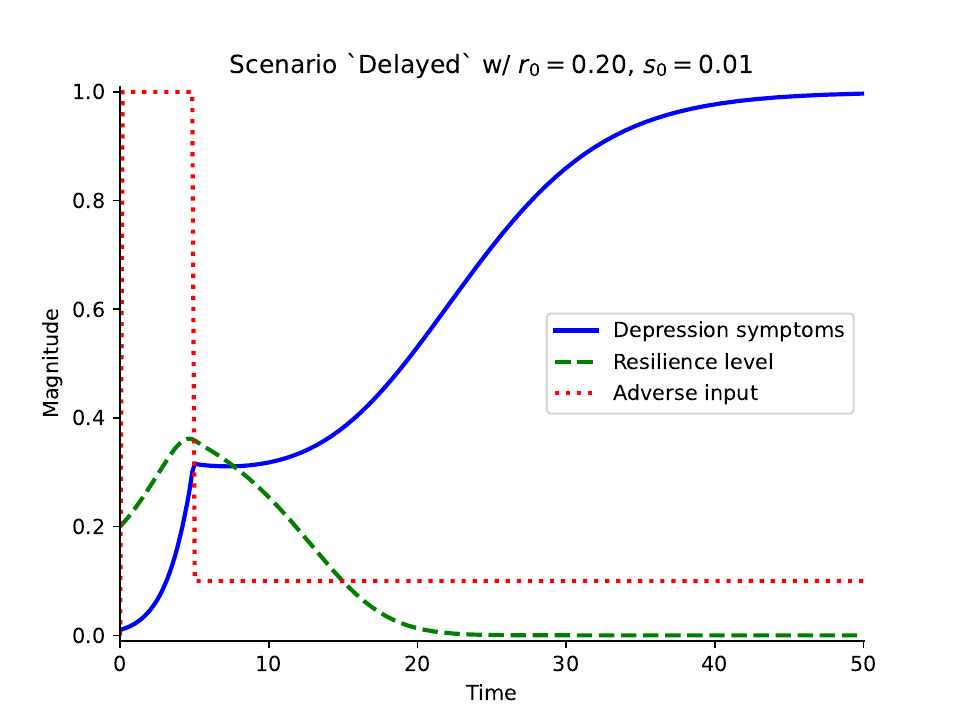}
  \caption{\label{fig:delayed-scenario}%
    \textbf{Delayed scenario.}
    A short adverse event during $t\in[0,5]$ triggers a rise of
    depression symptom from near zero. The resilience level starts out
    at $0.2$ and continues to improve; however, at around $t=5$, while
    the adverse input drops to a persistent but low $0.1$, the
    resilience and symptom levels have reached a tipping point: from
    here the symptom level continues to rise while the resilience
    level starts to deplete. Ultimately both levels approach their
    respective worst extremes (the \emph{despair} state).}
\end{figure}

\subsubsection*{The recovery scenario}

Despite the same adverse event as in the first (chronic) example, in
the \emph{recovery scenario} we see a patient who initially shows
worsening depression symptoms, which, however, then dissipate over
time, despite a persistent background adversity level, as is seen in
Fig~\ref{fig:recovery-scenario}.

\begin{figure}[htb]
  \centering
  \includegraphics[width=0.8\linewidth]{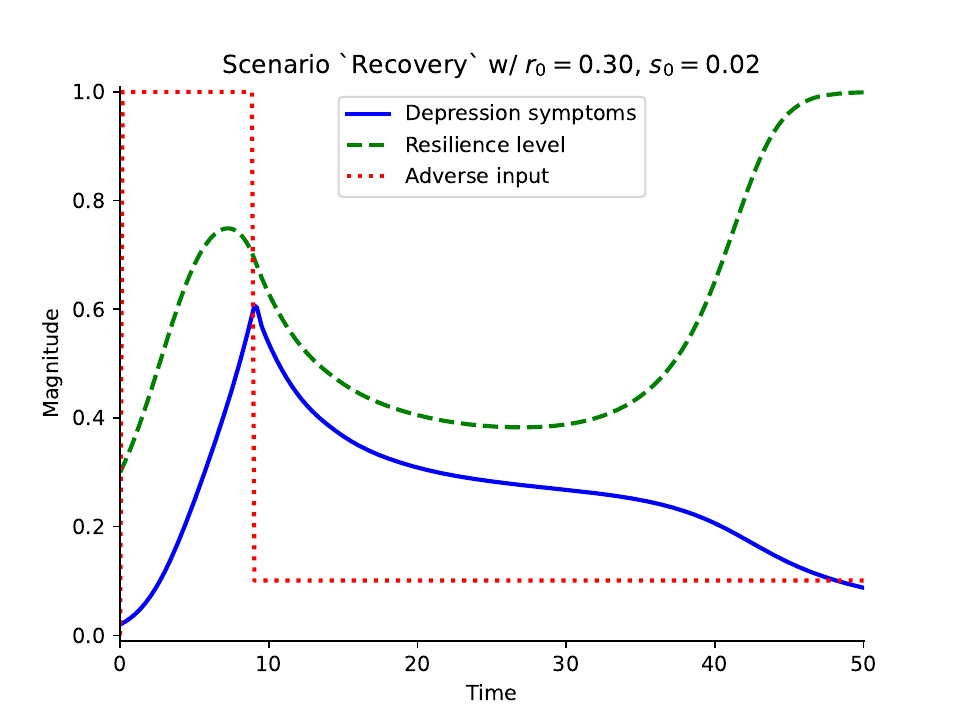}
  \caption{%
    \label{fig:recovery-scenario}%
    \textbf{Recovery scenario.}
    The adverse input, of intensity $1.0$ during
    $t\in[0,9]$, drops back to a persistent background adversity level
    of about $0.1$ afterwards. In this initial phase the depression symptom
    worsens from near zero while for the most part it is still small
    enough that the resilience level continues to recover, starting
    from $0.3$. However, a symptom level threshold is reached
    before $t=9$, so that the resilience level starts to decrease
    again.  When the adverse input drops to the background level at
    $t=9$, the depression symptom starts to improve and, while
    initially still declining, the resilience level eventually starts
    to grow again.}
\end{figure}

\subsubsection*{The resilience scenario}

Facing the same adversity as in previous examples, the
\emph{resilience scenario} demonstrates a patient whose resilience
level starts high and even grows further during the especially adverse
initial phase, while the depression symptom remains suppressed near
zero the entire time.

\begin{figure}[htb]
  \centering
  \includegraphics[width=0.8\linewidth]{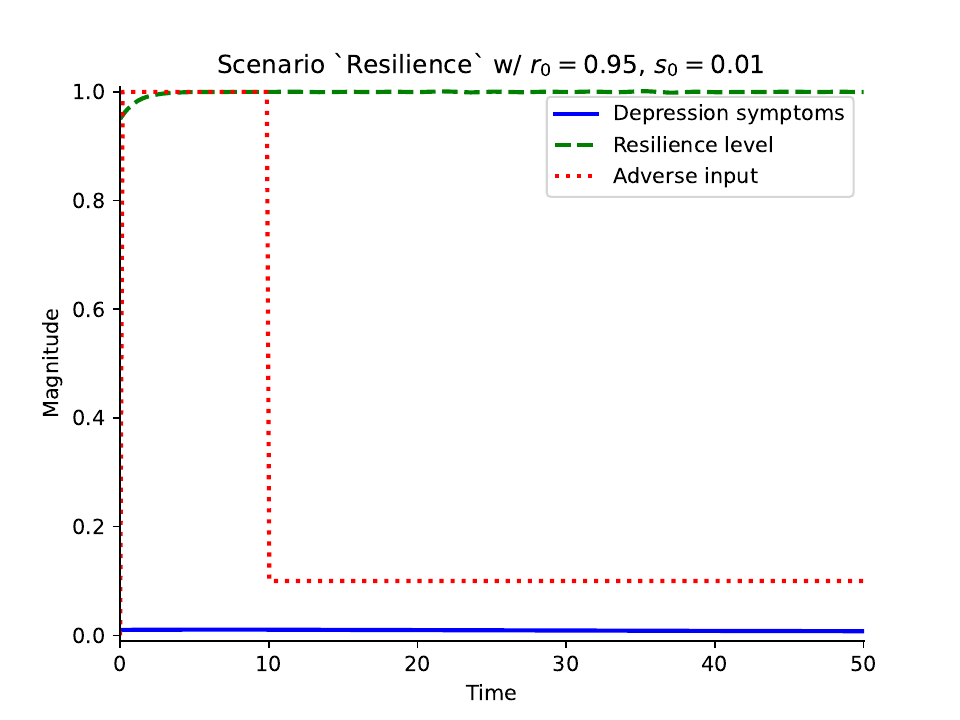}
  \caption{%
    \label{fig:resilience-scenario}%
    \textbf{Resilience scenario.}
    As in the chronic example, the adverse input, of
    intensity $1.0$ during $t\in[0,10]$, drops back to a persistent
    background adversity level of $0.1$ afterwards. The initial depression
    symptom level is near zero, however, the initial resilience level
    is near $1.0$. Despite the adverse input, the resilience level
    increases further and the depression symptom level remains near
    zero.}
\end{figure}

\subsubsection*{Improvement from pre-existing symptoms}

This example shows a patient with initially high depression symptoms
but also high resilience, who is not (or no longer) affected by
adversity, and whose depression symptoms thus dissipate with time,
cf.\ Fig~\ref{fig:pre-existing-symptom-improvement}.

\begin{figure}[htb]
  \centering
  \includegraphics[width=0.8\linewidth]{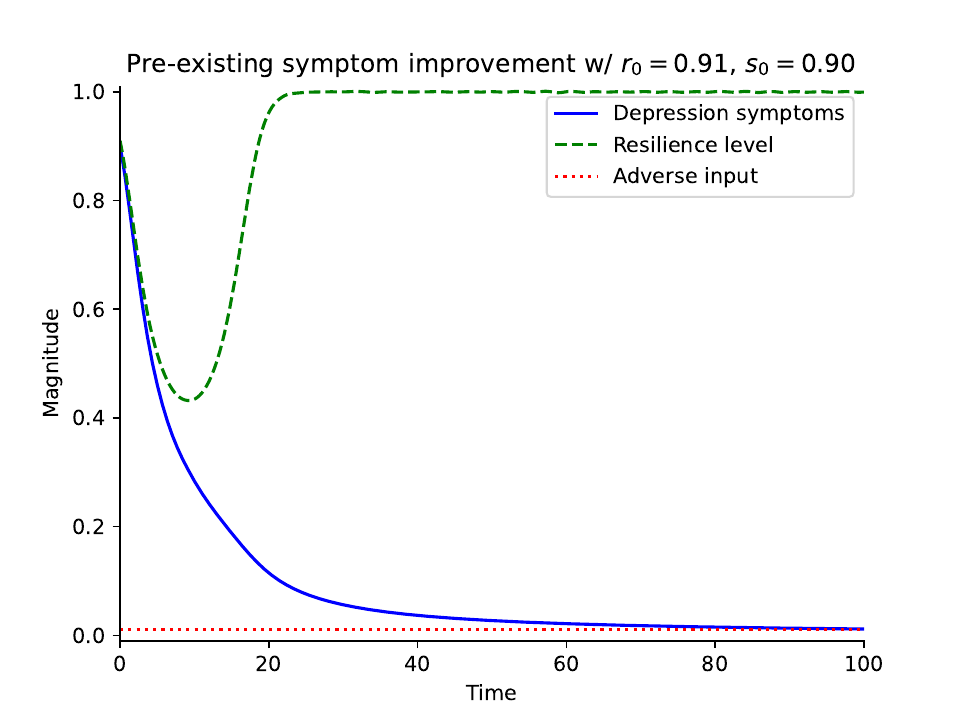}
  \caption{%
    \label{fig:pre-existing-symptom-improvement}%
    \textbf{Improvement from pre-existing symptoms scenario.}
    A patient with initially strong resilience but also strong
    depression symptoms may recover if no longer subjected to
    adversity.}
\end{figure}

\subsubsection*{Burnout from persistent low-grade adversity}

While in most of our examples the adverse input has the character of
an ``event'' that triggers the depression symptom to rise, depression
can also be a consequence of prolonged exposure to a seemingly small
adverse input. This is a feature commonly referred to as burnout,
which is usually interpreted as an accumulation of stress or adversity
over a long time horizon, depleting resilience and also causing a
depression, cf.\ Fig~\ref{fig:burnout}.

\begin{figure}[htb]
  \centering
  \includegraphics[width=0.8\linewidth]{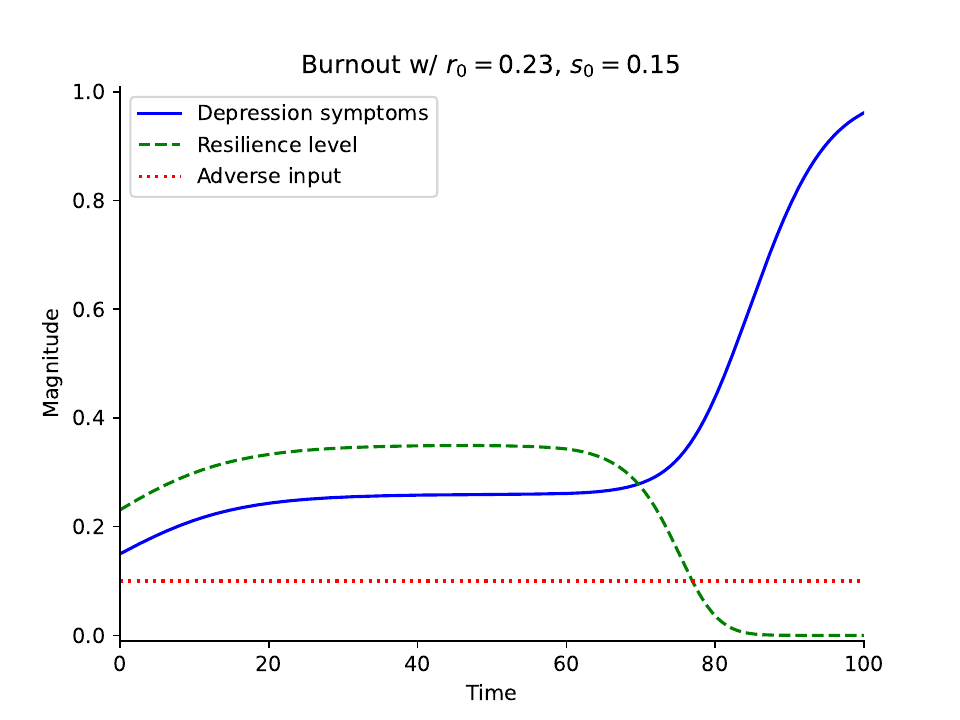}
  \caption{%
    \label{fig:burnout}%
    \textbf{Burnout from persistent low-grade adversity scenario.}
    This patient starts with a low level of depression symptoms while
    being subjected to an ongoing but low level of adversity. The
    initially low level of resilience increases during an initial
    phase, but with the growing depression symptom it eventually
    starts to decline, thereby in turn accelerating the descent into
    depression.  }
\end{figure}

\subsubsection*{Trigger events and relapse}

This example showcases a patient who is already depressed, but whose
symptoms seem to improve despite a level of background adversity. In
one scenario, the patient is subjected to a seemingly mild trigger
event, which, however, causes a relapse, cf.\
Fig~\ref{fig:trigger-relapse}.

\begin{figure}[htb]
  \centering
  \includegraphics[width=0.49\linewidth]{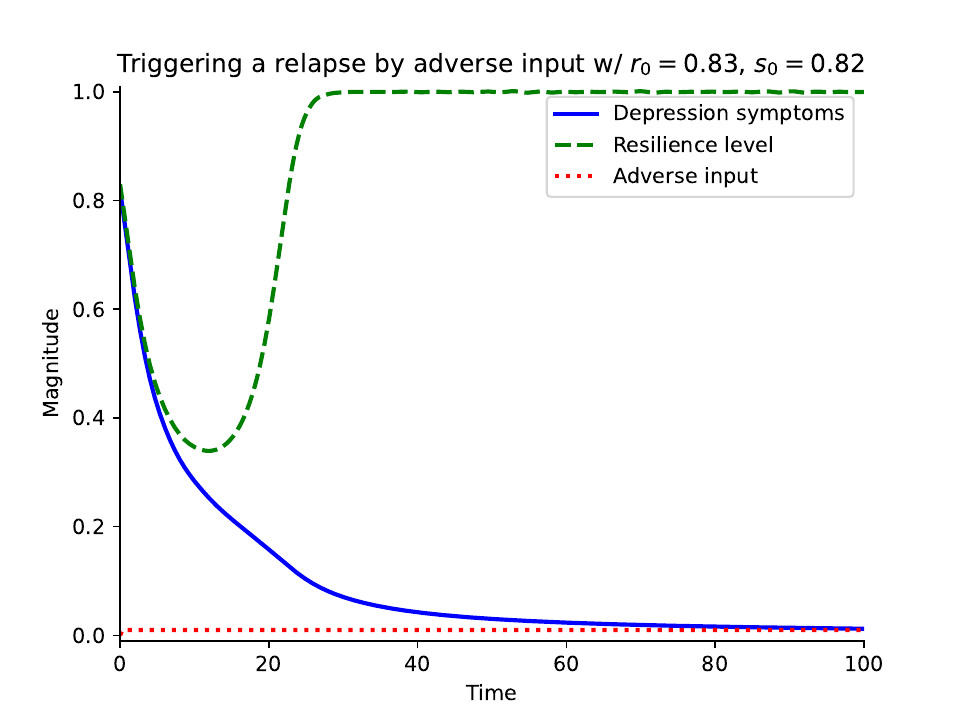}\hfil%
  \includegraphics[width=0.49\linewidth]{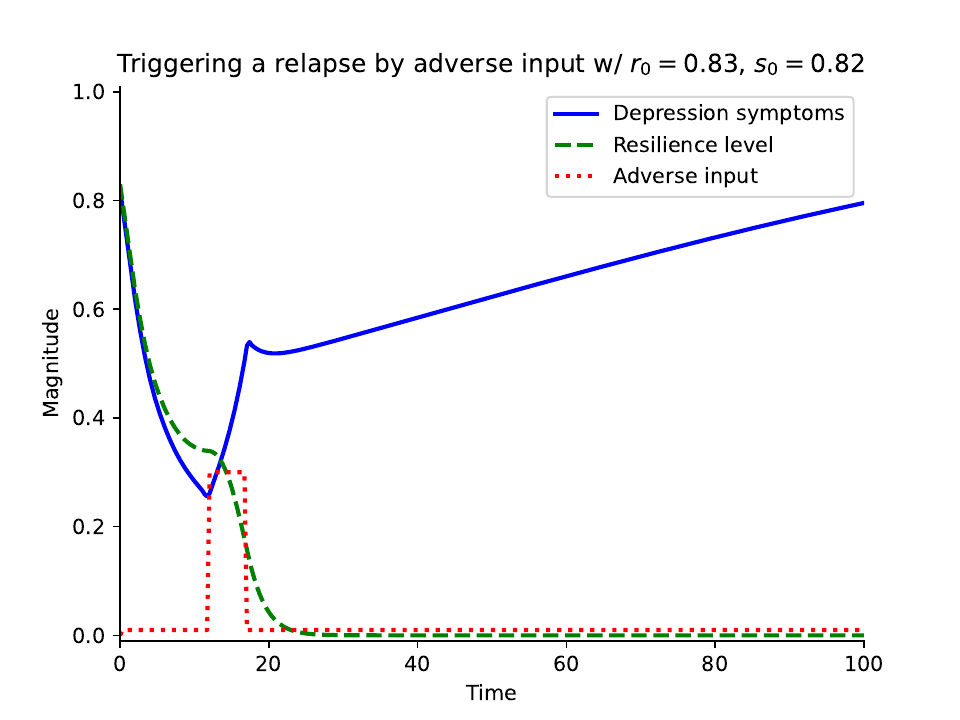}
  \caption{%
    \label{fig:trigger-relapse}%
    \textbf{Relapse scenario.}
    We see two almost identical scenarios of an initially recovering
    patient (until time $t=12$, that is). In one scenario the patient
    is subjected to an adverse trigger event between $t=12$ and $t=17$
    of intensity $0.3$, which causes relapse. At all other times there
    is an almost zero adversity level of $0.01$ affecting the patient.}
\end{figure}

\subsubsection*{Recurrent adversities triggering multiple depressive episodes}

Trigger events may recur, and anecdotal evidence available to the
authors suggests that depressive episodes come in waves, often
following trigger events, with symptoms improving in between. Such a scenario can be demonstrated for model~\eqref{eq:3-RS},
cf.\ Fig~\ref{fig:multiple-adversities-and-depression-episodes},
where a periodic adverse input has been used. It should be noted that
the input does not have to be periodic to observe similar improvement
and relapse periods.

\begin{figure}[htb]
  \centering
  \includegraphics[width=0.8\linewidth]{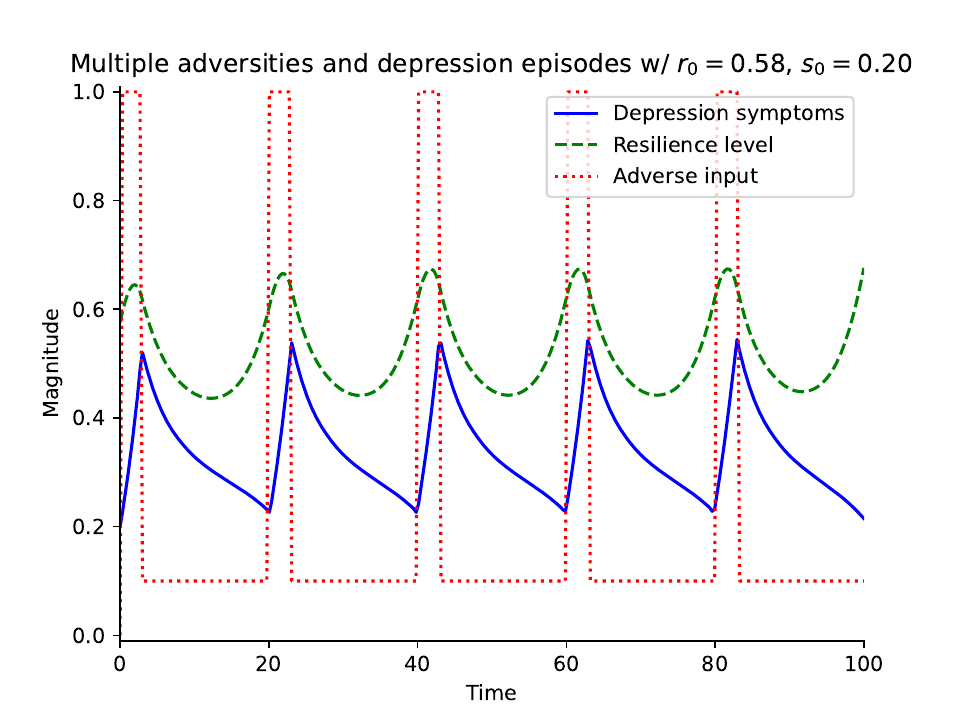}
  \caption{%
    \label{fig:multiple-adversities-and-depression-episodes}%
    \textbf{Recurrent adversities triggering multiple depressive episodes.}
    If adverse events followed by less stressful times are a recurrent
    theme, a patient's depression may worsen and improve with these
    phases. It should be noted that this oscillatory behaviour is
    unstable and a patient would most likely either recover or turn to
    a severe depression eventually.}
\end{figure}

\subsection*{Structural properties of the model}

\subsubsection*{Monotonicity of trajectories}

This example demonstrates the monotonicity properties of
system~\eqref{eq:3-RS}: Ordered initial conditions and inputs lead to
ordered trajectories, cf.\ Fig~\ref{fig:monotonicity}.

\begin{figure}[htb]
  \centering
  \includegraphics[width=0.8\linewidth]{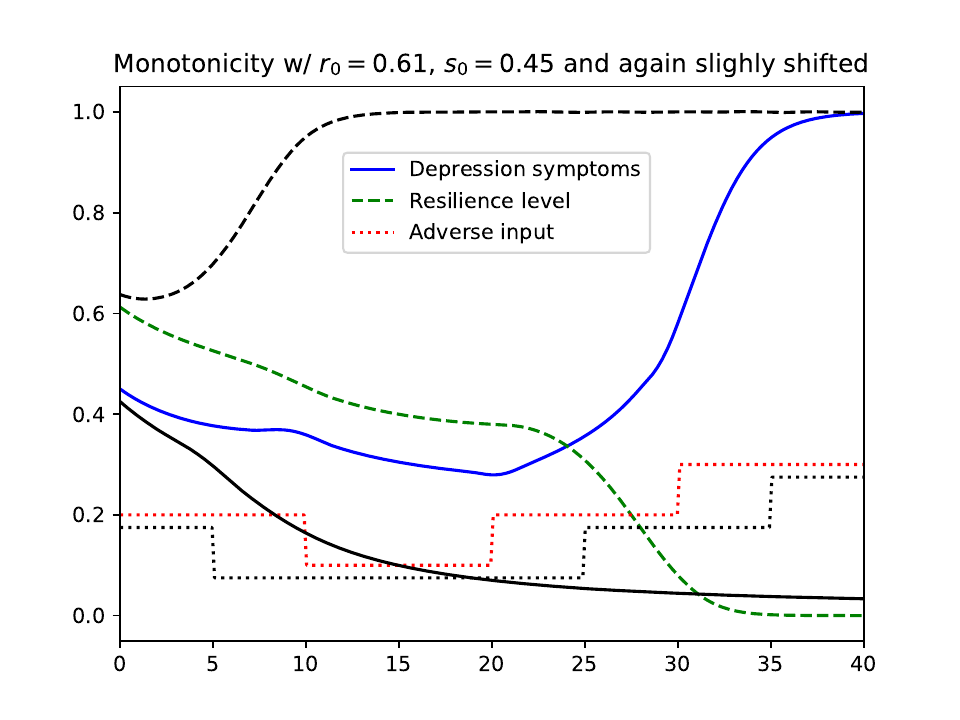}
  \caption{%
    \label{fig:monotonicity}%
    \textbf{Monotonicity of trajectories.}
    Two evolutions of the resilience level, symptom level, and adverse
    input are shown; one in color and one in black.  The black version
    has initial conditions that are ``not worse than'' those of the
    colored version, each being shifted by $0.025$ in the appropriate
    direction. Likewise, the adverse input in the black scenario is
    ``not worse than'' the respective input for the colored scenario,
    being (here even strictly) smaller at every time instant. The
    resulting $r(t)$, $s(t)$ trajectories stay ordered as well, that
    is, the resilience level of the black version stays above its
    colored counterpart, and the black symptom curve below the colored
    symptom curve.}
\end{figure}

\subsubsection*{The resilience level and fate}

Starting from the same depression symptom level and subjected to the
same adverse inputs, it is possible to ultimately reach bliss or
despair, depending on two different yet close initial resilience levels,
cf.\ Fig~\ref{fig:fate-by-resilience-level}.

\begin{figure}[htb]
  \centering
  \includegraphics[width=0.49\linewidth]{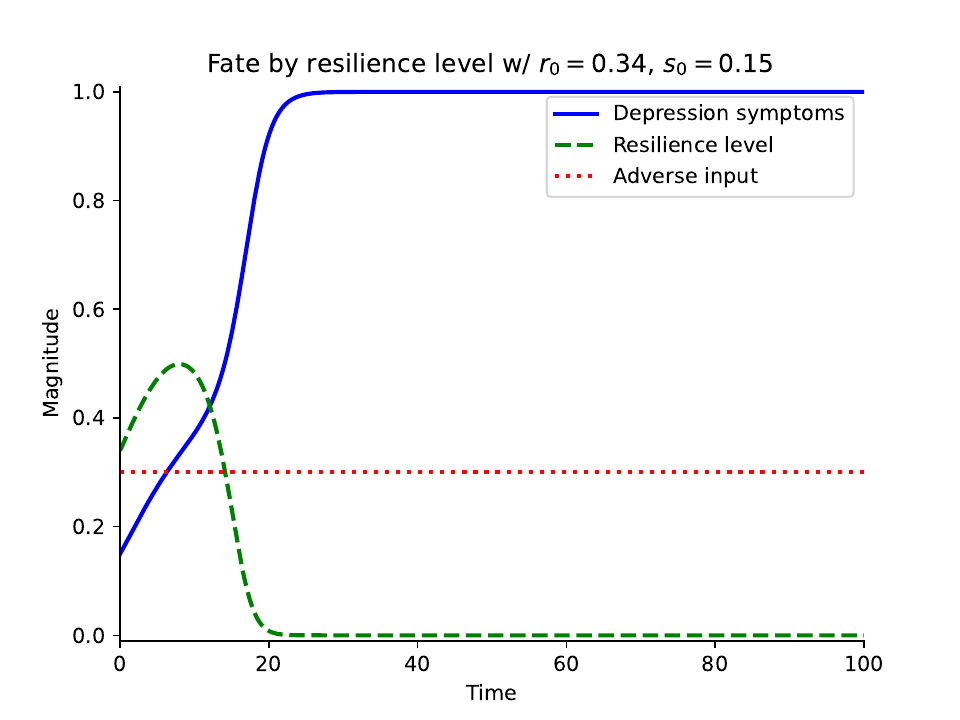}\hfil%
  \includegraphics[width=0.49\linewidth]{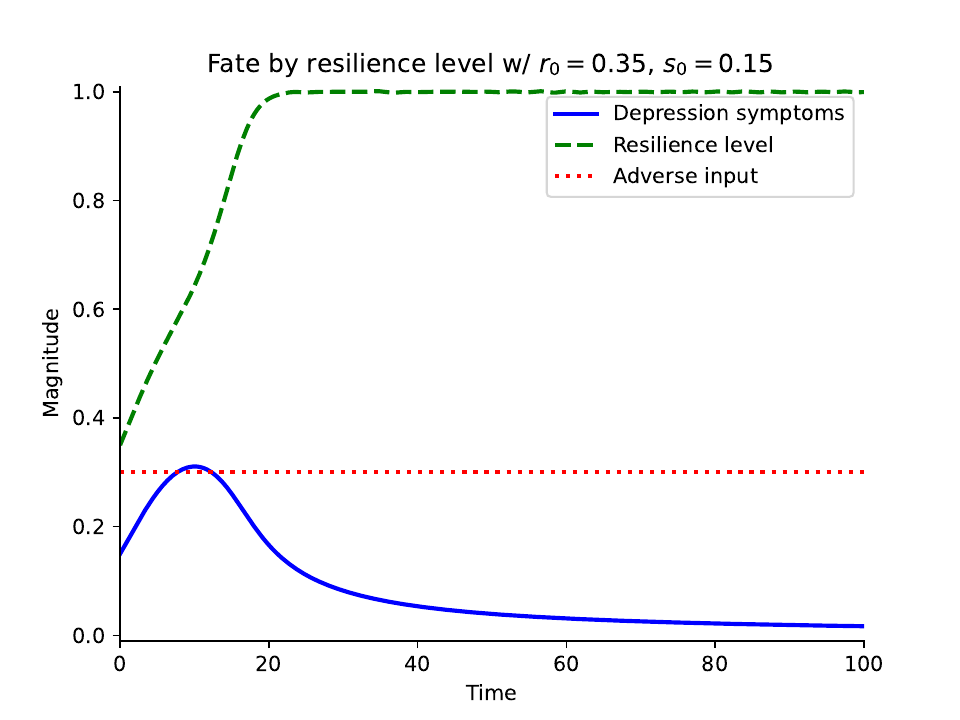}
  \caption{%
    \label{fig:fate-by-resilience-level}%
    \textbf{Resilience level and fate.}
    We see two trajectories starting at the same depression symptom
    level. Both are affected by the same adverse input (which is
    constant for simplicity). The only difference in the parameters is
    that one has a slightly higher initial resilience level than the
    other. It is this one that converges to bliss, while the other
    converges to despair.}
\end{figure}

\subsubsection*{The depression symptom level and fate}

Similar to the previous example, starting from the same resilience
level and subjected to the same adverse inputs, it is possible to
ultimately reach bliss or despair, depending on two different yet close
initial depression symptom levels, as is best seen using a phase plot,
cf.\ Fig~\ref{fig:fate-by-depression-symptom-level}.

\begin{figure}[htb]
  \centering
  \includegraphics[width=0.49\linewidth]{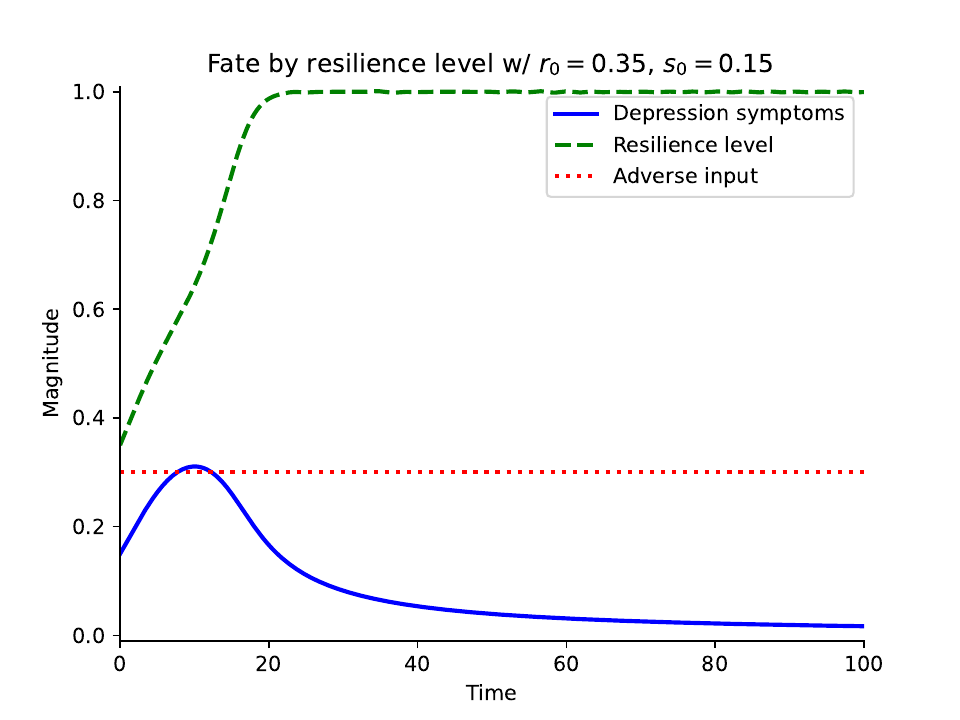}\hfil%
  \includegraphics[width=0.49\linewidth]{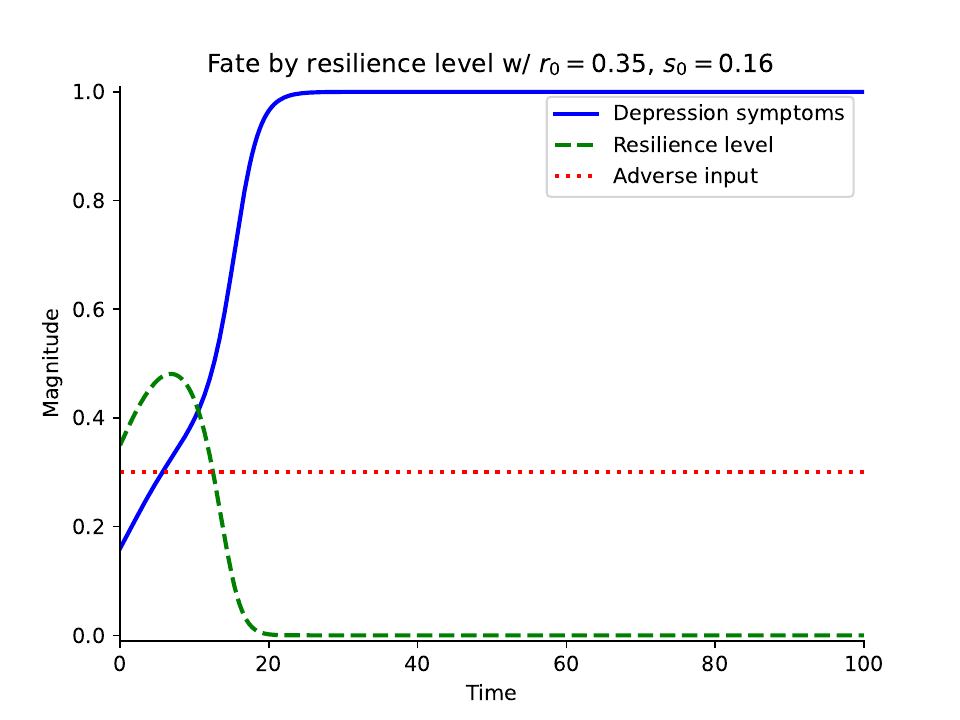}
  \caption{%
    \label{fig:fate-by-depression-symptom-level}%
    \textbf{Depression symptom level and fate.}
    We see two trajectories starting at the same resilience
    level. Both are affected by the same adverse input (which is
    constant for simplicity). The only difference in the parameters is
    that one has a slightly lower initial depression symptom than the
    other. It is this one that converges to bliss, while the other
    converges to despair.}
\end{figure}

\subsubsection*{Alternative pasts}

It is entirely possible that a patient with specific resilience and
symptom levels at a given time may have arrived at this state from
very different pasts, cf.\ Fig~\ref{fig:alternative-pasts}.
\begin{figure}[htb]
  \centering
  \includegraphics[width=0.8\linewidth]{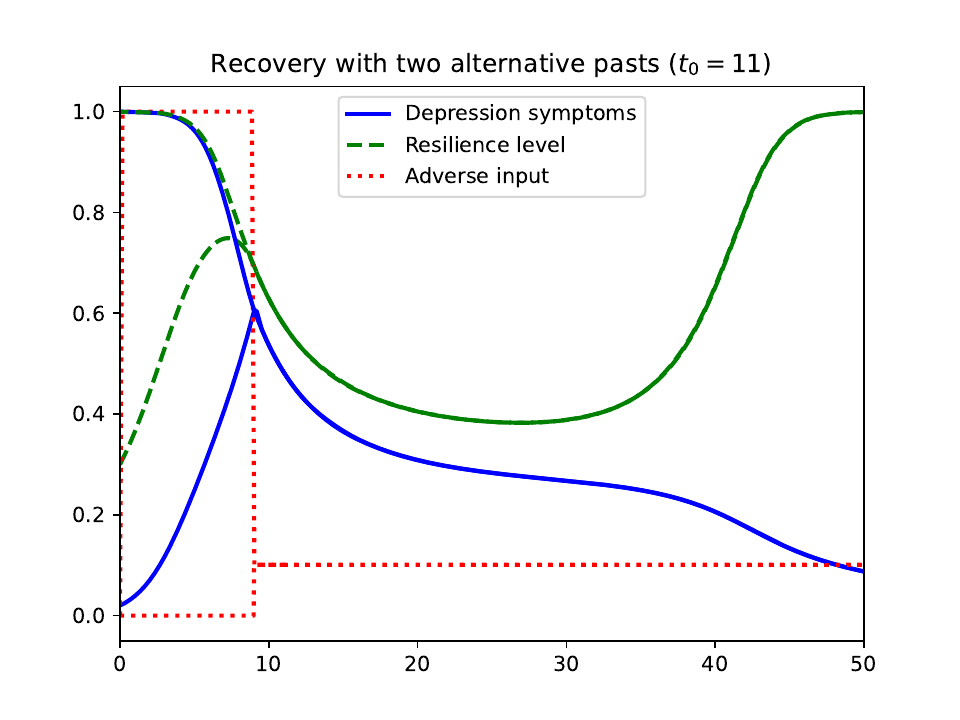}
  \caption{%
    \label{fig:alternative-pasts}%
    \textbf{Different pasts, same futures at $t=11$.}
    Different pasts can lead to the same futures. Here input,
    resilience, and symptom trajectories coincide for $t\geq 11$
    (consider this the initial time, not $t=0$). However, before this
    instant, in one possible past the patient started out quite
    resilient yet very depressed and was affected by no adverse
    input. In another possible past, the patient started with a
    resilience of about $0.3$, a negligible depression symptom, but
    was affected by adverse inputs of about $1.0$. In both cases the
    patient seems to recover from the depression and regain resilience
    over time, while remaining subjected to a small but consistent
    adverse input of about $0.1$.}
\end{figure}

\subsubsection*{Sustained low-grade symptoms}

The model allows for sustained low-grade symptoms in the presence of a
non-vanishing input. Numerically, this is a challenging scenario
though, as system~\eqref{eq:3-RS} does not admit any \emph{stable}
interior equilibria inside $[0,1]^{2}$, and computing a trajectory
forward in time that approaches an unstable equilibrium is extremely
sensitive to initial conditions, inputs, and the numerical integration
method. In practice, such a trajectory would be unsustainable for
longer periods of time with this model, and it would be more realistic
that a patient's depression eventually either worsens or improves,
or that some external feedback mechanism is at play that provides
stabilization. It is also unrealistic that an adverse input would
simply remain constant, and a state-based (read: feedback-controlled)
input might achieve a scenario similar to this one.

Due to these mathematical limitations, the present example has
been constructed by integrating the model backward in time, starting from
the equilibrium. By uniqueness of solutions, this trajectory is,
at least in theory, also obtainable by integrating forward in time,
cf.\ Fig~\ref{fig:sustained-low-grade-symptoms}.

\begin{figure}[htb]
  \centering
  \includegraphics[width=0.8\linewidth]{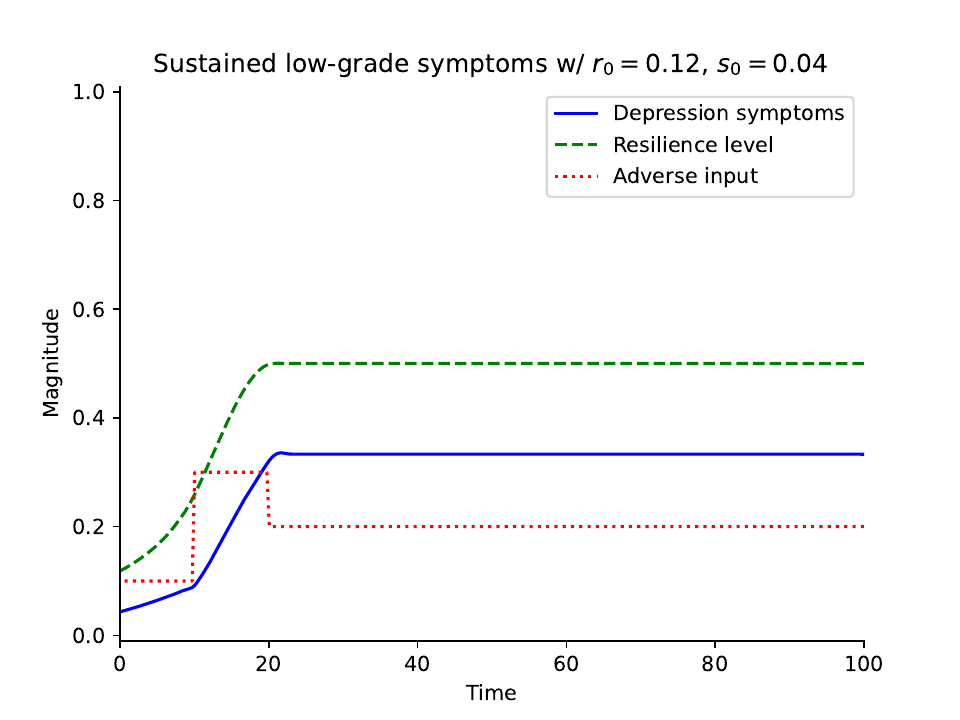}
  \caption{%
    \label{fig:sustained-low-grade-symptoms}%
    \textbf{Sustained low-grade symptoms: the interior equilibrium $p_{5}=p_{5}(e)$.}
    This figure demonstrates a patient approaching near constant
    but sustained symptoms, corresponding to the unstable equilibrium
    $p_{5}$. We note that for numerical reasons this plot has been
    created by integrating the system backward in time.}
\end{figure}

Another possibility to have constant, low-grade symptoms is via the
equilibria $P_{6}$ and $P_{7}$ for constant adverse inputs $e=0$
and, respectively, $e=1$, cf.\
Fig~\ref{fig:sustained-low-grade-symptoms-equilibrium}.

\begin{figure}[htb]
  \centering
  \includegraphics[width=0.49\linewidth]{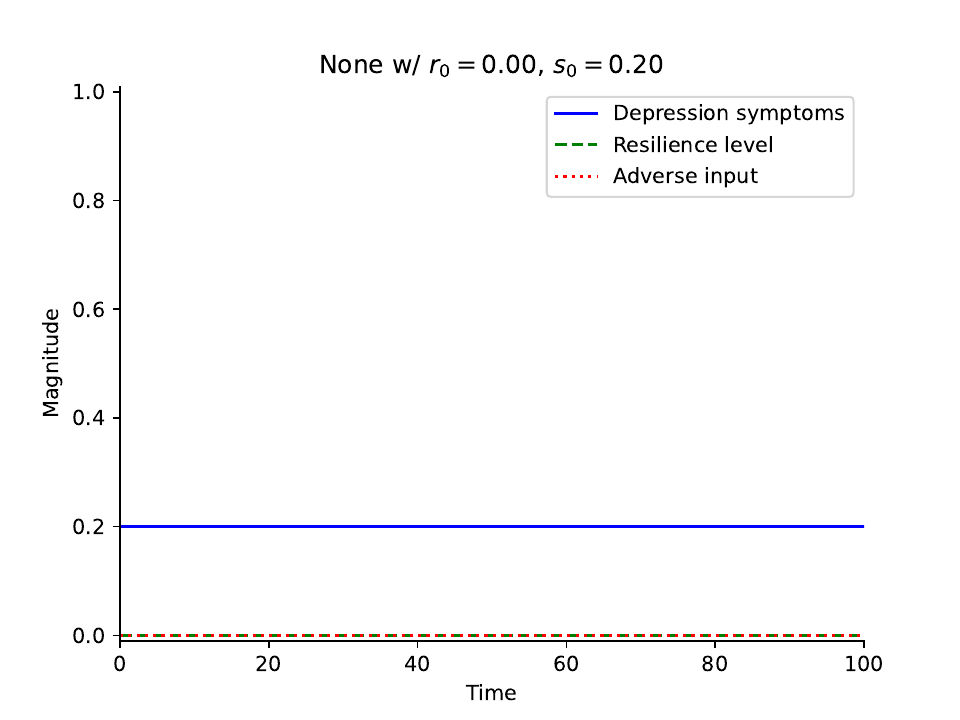}\hfil%
  \includegraphics[width=0.49\linewidth]{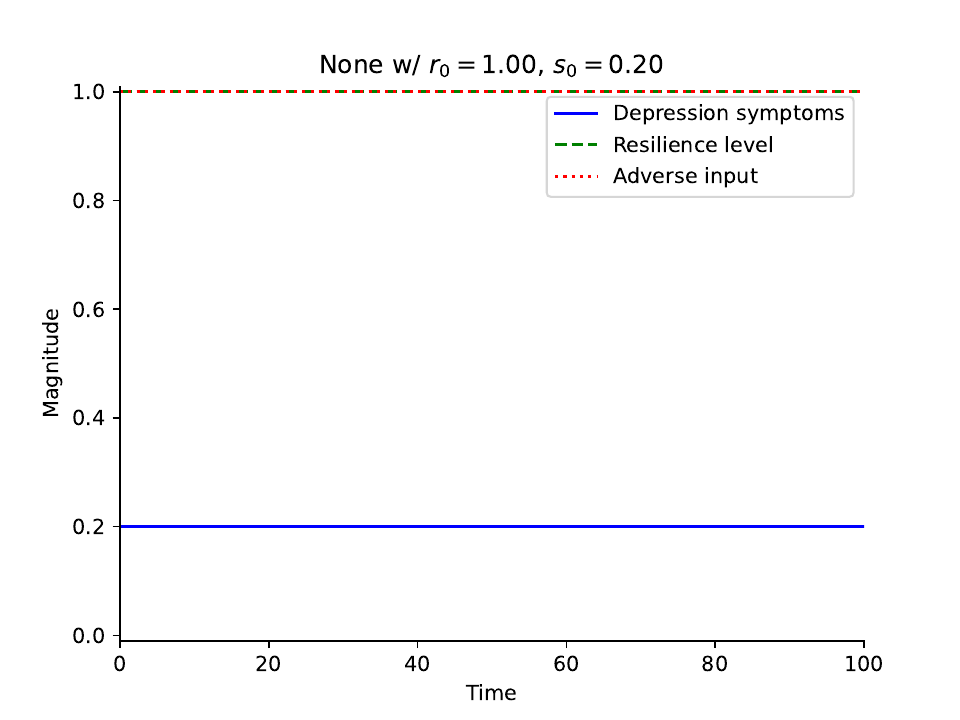}
  \caption{%
    \label{fig:sustained-low-grade-symptoms-equilibrium}%
    \textbf{Sustained low-grade symptoms: boundary equilibria $P_{6}$ and $P_{7}$.}
    These two plots showcase persistent symptom levels for constant
    inputs, corresponding to trajectories starting in the equilibrium
    sets $P_{6}$ with constant $e=0$ and $P_{7}$ with constant $e=1$.}
\end{figure}

\subsubsection*{The universal threshold $r=s$}

Following
Theorem~\ref{thm:stability}\ref{item:RS-upper-triangle-invariance-and-ROA}
we stated that patients cannot recover once the symptom level
reaches the resilience level, i.e., once $s\geq r$, as long as the
adverse input does not vanish. We demonstrate this with a highly
resilient patient whose symptom level, however, matches the resilience
level, $r_{0}=s_{0}=0.95$, and who is subjected to an almost
negligible constant adversity $e=0.01$: the patient will nevertheless
approach severe depression eventually, cf.\
Fig~\ref{fig:universal-threshold-at-r-s}.

\begin{figure}[htb]
  \centering
  \includegraphics[width=0.8\linewidth]{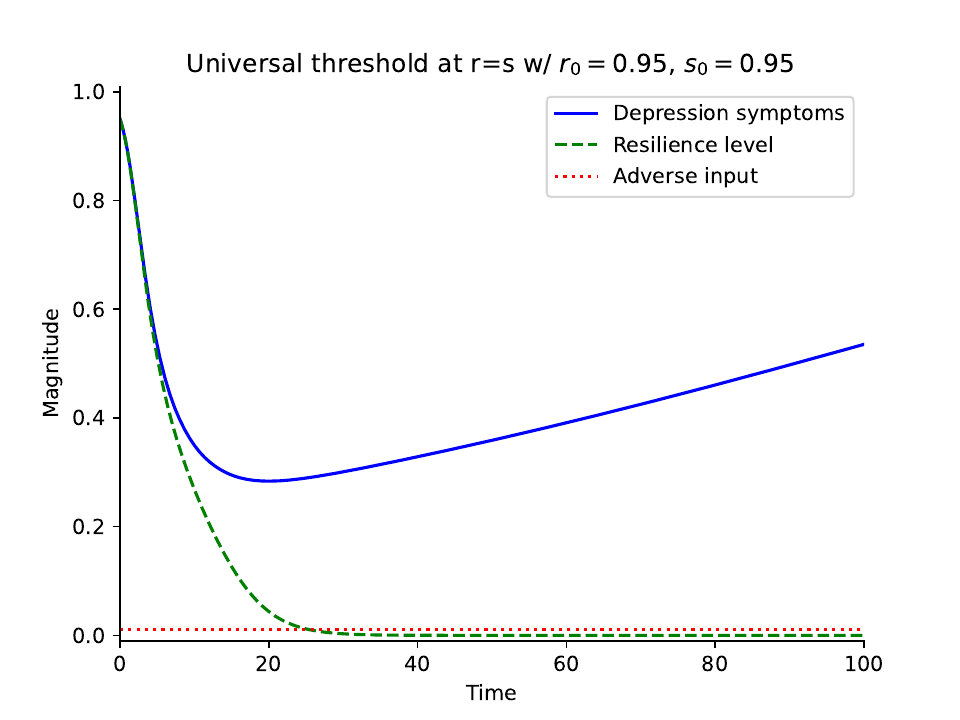}
  \caption{%
    \label{fig:universal-threshold-at-r-s}%
    \textbf{Universal threshold for depression onset at $r=s$.}
    If $s_{0}\geq r_{0}$ (here $r_{0}=s_{0}=0.95$), then by
    Theorem~\ref{thm:stability} any adverse input that does not vanish
    ---here the almost negligible constant $e=0.01$---drives the
    patient into despair: the resilience level depletes, and the
    symptom level, after an initial decline, slowly approaches its
    worst extreme.}
\end{figure}

\section*{Further experiments}

The reader is invited to experiment with these examples on their own
using an interactive simulation environment with preconfigured
scenarios, which the authors provide as open-source software under the
MIT license~\cite{rueffer-schoenlein-rsmodel-software}. It is
implemented as a \texttt{marimo} notebook for Python~3.12 or newer and
builds on the Python package \texttt{rsmodel}; the same repository also
contains the symbolic computations underpinning the results of this
article and the scripts that regenerate every figure in it. A version of
the notebook that runs without any installation in the user's web
browser is available at
\begin{center}
  \url{https://rsmodel.org}
\end{center}
(no technical knowledge required for basic usage). The corresponding
source code is archived under
\href{https://doi.org/10.5281/zenodo.22018288}{doi:10.5281/zenodo.22018288}
and developed at \url{https://github.com/bjoseru/rsmodel}.

\section*{Conclusions and outlook}

We provide a dynamic model for depression with two state variables
\textit{resilience} and \textit{symptom}. The model is a system of two
coupled ODEs. The right-hand side of this model is a polynomial of total degree
four (and of degree at most three in each variable). This basic mathematical model is able to recover a range of
features that seem compatible with anecdotal evidence observed in
individuals suffering from depression and with trajectories showcased
in the resilience research literature.

The model follows the paradigm of being as simple as possible, but not
simpler. It supports observations that a patient should continue to
take their medication despite already feeling better---as the internal
resilience/memory state may not have recovered yet, and as this may take
substantially longer. It also supports scenarios where
a patient is catapulted back into a severe depression by seemingly
small trigger events.

At the same time, the modeling process used here has made no attempt
to seek grounding in more elaborate models of schema therapy or causal
networks of signaling pathways in the brain as they are studied in
neuroscience. The objective here was merely to provide a model that
captures qualitative behavior, and we argue that this objective has
been achieved.

A number of further extensions are possible. On the mathematical side,
more analysis can be done, for example by computing regions of
attraction, by considering stability notions such as input-to-state
(dynamical) stability~\cite{isds}, or by designing observers for the
memory state using only measurements of inputs and possibly quantized
measurements of the symptom state. The model could further be
augmented to account for treatment options such as medication, or to
accommodate multiple symptoms and/or memory features, even though this
would be a departure from the minimalist aim expressed by the authors.

\def\cprime{$'$}


\begin{thebibliography}{10}

\bibitem{world-health-organization2023-depressive-disorder-depression}
{World Health Organization}.
\newblock Depressive disorder (depression) [Fact sheet]; 2025 [cited 2026
  August 11].
\newblock Available from:
  \url{https://www.who.int/news-room/fact-sheets/detail/depression}.

\bibitem{byrumahearnkrishnan1999-a-neuroanatomic-model-for-depression}
Byrum CE, Ahearn EP, Krishnan KRR.
\newblock A neuroanatomic model for depression.
\newblock Progress in Neuro-Psychopharmacology and Biological Psychiatry.
  1999;23(2):175-93.
\newblock \href {http://dx.doi.org/10.1016/S0278-5846(98)00106-7}
  {doi:10.1016/S0278-5846(98)00106-7}.

\bibitem{disnerbeevershaighbeck2011-neural-mechanisms-of-the-cognitive-model-of-depression}
Disner SG, Beevers CG, Haigh EAP, Beck AT.
\newblock Neural mechanisms of the cognitive model of depression.
\newblock Nature Reviews Neuroscience. 2011 Aug;12(8):467-77.
\newblock \href {http://dx.doi.org/10.1038/nrn3027} {doi:10.1038/nrn3027}.

\bibitem{demiccheng2014-modeling-the-dynamics-of-disease-states-in-depression}
Demic S, Cheng S.
\newblock Modeling the Dynamics of Disease States in Depression.
\newblock PLOS ONE. 2014 Oct;9(10):1-14.
\newblock \href {http://dx.doi.org/10.1371/journal.pone.0110358}
  {doi:10.1371/journal.pone.0110358}.

\bibitem{chengdorsognachou2020-mathematical-modeling-of-depressive-disorders:-circadian-driving-bistability-and-dynamical-transitions}
Cheng X, D'Orsogna MR, Chou T.
\newblock Mathematical modeling of depressive disorders: Circadian driving,
  bistability and dynamical transitions.
\newblock Computational and Structural Biotechnology Journal. 2021;19:664-90.
\newblock \href {http://dx.doi.org/10.1016/j.csbj.2020.10.035}
  {doi:10.1016/j.csbj.2020.10.035}.

\bibitem{bothhoogendoornkleintreur2008-modeling-the-dynamics-of-mood-and-depression}
Both F, Hoogendoorn M, Klein M, Treur J.
\newblock Modeling the Dynamics of Mood and Depression.
\newblock In: Proceedings of the 2008 Conference on ECAI 2008: 18th European
  Conference on Artificial Intelligence; 2008. p. 266-70.

\bibitem{tuckwellmiura1978-a-mathematical-model-for-spreading-cortical-depression}
Tuckwell HC, Miura RM.
\newblock A mathematical model for spreading cortical depression.
\newblock Biophysical Journal. 1978;23(2):257-76.
\newblock \href {http://dx.doi.org/10.1016/S0006-3495(78)85447-2}
  {doi:10.1016/S0006-3495(78)85447-2}.

\bibitem{liebrigottischafer2025-ausgebrannt-sein:-burnout-als-risikozustand}
Lieb K, Rigotti T, Sch\"afer S.
\newblock Ausgebrannt sein: Burnout als Risikozustand.
\newblock Forschung und {L}ehre. 2025;1$\vert$25:36-9.

\bibitem{arnoldschilbachrigotti2023-paradigmen-der-psychologischen-resilienzforschung}
Arnold M, Schilbach M, Rigotti T.
\newblock Paradigmen der psychologischen Resilienzforschung.
\newblock Psychologische {R}undschau. 2023;74(3):154-65.
\newblock \href {http://dx.doi.org/10.1026/0033-3042/a000627}
  {doi:10.1026/0033-3042/a000627}.

\bibitem{galatzer-levyhuangbonanno2018-trajectories-of-resilience-and-dysfunction-following-potential-trauma:-a-review-and-statistical-evaluation}
Galatzer-Levy IR, Huang SH, Bonanno GA.
\newblock Trajectories of resilience and dysfunction following potential
  trauma: A review and statistical evaluation.
\newblock Clinical Psychology Review. 2018;63:41-55.
\newblock \href {http://dx.doi.org/10.1016/j.cpr.2018.05.008}
  {doi:10.1016/j.cpr.2018.05.008}.

\bibitem{schaferkunzlerkalischtuscherlieb2022-trajectories-of-resilience-and-mental-distress-to-global-major-disruptions}
Sch\"afer SK, Kunzler AM, Kalisch R, T\"uscher O, Lieb K.
\newblock Trajectories of resilience and mental distress to global major
  disruptions.
\newblock Trends in Cognitive Sciences. 2022;26(12):1171-89.
\newblock \href {http://dx.doi.org/10.1016/j.tics.2022.09.017}
  {doi:10.1016/j.tics.2022.09.017}.

\bibitem{bonanno2004-loss-trauma-and-human-resilience:-have-we-underestimated-the-human-capacity-to-thrive-after-extremely-aversive-events}
Bonanno GA.
\newblock Loss, Trauma, and Human Resilience: Have We Underestimated the Human
  Capacity to Thrive after Extremely Aversive Events?
\newblock American Psychologist. 2004;59(1):20-8.
\newblock \href {http://dx.doi.org/10.1037/0003-066X.59.1.20}
  {doi:10.1037/0003-066X.59.1.20}.

\bibitem{becksteerbrownothers1996-manual-for-the-beck-depression-inventory-ii}
Beck AT, Steer RA, Brown GK, et~al.
\newblock Manual for the {B}eck depression inventory-{II}.
\newblock San Antonio, TX: Psychological Corporation; 1996.
\newblock \href {http://dx.doi.org/10.1037/t00742-000}
  {doi:10.1037/t00742-000}.

\bibitem{becksteer1988-bhs-beck-hopelessness-scale:-manual}
Beck AT, Steer RA.
\newblock {BHS}, {B}eck hopelessness scale: manual.
\newblock Psychological Corporation, San Antonio, TX; 1988.

\bibitem{karin-et-al2020-new-hpa-model-dysregulation}
Karin O, Raz M, Tendler A, Bar A, Korem~Kohanim Y, Milo T, et~al.
\newblock A new model for the HPA axis explains dysregulation of stress
  hormones on the timescale of weeks.
\newblock Molecular Systems Biology. 2020;16(7):e9510.
\newblock \href {http://dx.doi.org/10.15252/msb.20209510}
  {doi:10.15252/msb.20209510}.

\bibitem{kalischbakeral.2017-the-resilience-framework-as-a-strategy-to-combat-stress-related-disorders}
Kalisch R, Baker DG, Basten U, et~al.
\newblock The resilience framework as a strategy to combat stress-related
  disorders.
\newblock Nature Human Behaviour. 2017;1:784-90.
\newblock \href {http://dx.doi.org/10.1038/s41562-017-0200-8}
  {doi:10.1038/s41562-017-0200-8}.

\bibitem{smith1995-monotone-dynamical-systems}
Smith HL.
\newblock Monotone dynamical systems. vol.~41 of Mathematical Surveys and
  Monographs.
\newblock Providence, RI: American Mathematical Society; 1995.

\bibitem{angelisontag2003-monotone-control-systems}
Angeli D, Sontag ED.
\newblock Monotone control systems.
\newblock IEEE Transactions on Automatic Control. 2003;48(10):1684-98.
\newblock \href {http://dx.doi.org/10.1109/TAC.2003.817920}
  {doi:10.1109/TAC.2003.817920}.

\bibitem{khalil2002-nonlinear-systems}
Khalil HK.
\newblock Nonlinear systems.
\newblock Prentice Hall; 2002.

\bibitem{hinrichsenpritchard2005-mathematical-systems-theoryi---modelling-state-space-analysis-stability-and-robustness}
Hinrichsen D, Pritchard AJ.
\newblock Mathematical Systems Theory~I---Modelling, State Space Analysis,
  Stability and Robustness.
\newblock Berlin: Springer; 2005.
\newblock \href {http://dx.doi.org/10.1007/b137541} {doi:10.1007/b137541}.

\bibitem{Aman-1990-ODE-engl}
Amann H.
\newblock Ordinary Differential Equations---An Introduction to Nonlinear
  Analysis.
\newblock Berlin: Walter de Gruyter; 1990.
\newblock \href {http://dx.doi.org/10.1515/9783110853698}
  {doi:10.1515/9783110853698}.

\bibitem{HirschSmith2005-handbook}
Hirsch MW, Smith H.
\newblock Monotone Dynamical Systems.
\newblock In: Ca{\~n}ada A, Dr{\'a}bek P, Fonda A, editors. Handbook of
  Differential Equations: Ordinary Differential Equations. vol.~2. Amsterdam:
  Elsevier; 2005. p. 239-357.
\newblock \href {http://dx.doi.org/10.1016/S1874-5717(05)80006-9}
  {doi:10.1016/S1874-5717(05)80006-9}.

\bibitem{youngs.weishaar2003-schema-therapy:-a-practitioners-guide}
Young JE, Klosko JS, Weishaar ME.
\newblock Schema Therapy: A Practitioner's Guide.
\newblock Guilford Publications; 2003.

\bibitem{Segal1996}
Segal ZV, Williams JM, Teasdale JD, Gemar M.
\newblock A cognitive science perspective on kindling and episode sensitization
  in recurrent affective disorder.
\newblock Psychological Medicine. 1996;26(2):371-80.
\newblock \href {http://dx.doi.org/10.1017/S0033291700034760}
  {doi:10.1017/S0033291700034760}.

\bibitem{Post1992}
Post RM.
\newblock Transduction of psychosocial stress into the neurobiology of
  recurrent affective disorder.
\newblock American Journal of Psychiatry. 1992;149(8):999-1010.
\newblock \href {http://dx.doi.org/10.1176/ajp.149.8.999}
  {doi:10.1176/ajp.149.8.999}.

\bibitem{Monroe2005}
Monroe SM, Harkness KL.
\newblock Life stress, the ``kindling'' hypothesis, and the recurrence of
  depression: Considerations from a life stress perspective.
\newblock Psychological Review. 2005;112(2):417-45.
\newblock \href {http://dx.doi.org/10.1037/0033-295X.112.2.417}
  {doi:10.1037/0033-295X.112.2.417}.

\bibitem{McEwenStellar1993}
McEwen BS, Stellar E.
\newblock Stress and the individual: Mechanisms leading to disease.
\newblock Archives of Internal Medicine. 1993;153(18):2093-101.
\newblock \href {http://dx.doi.org/10.1001/archinte.1993.00410180039004}
  {doi:10.1001/archinte.1993.00410180039004}.

\bibitem{McEwen1998}
McEwen BS.
\newblock Stress, adaptation, and disease: Allostasis and allostatic load.
\newblock Annals of the New York Academy of Sciences. 1998;840:33-44.
\newblock \href {http://dx.doi.org/10.1111/j.1749-6632.1998.tb09546.x}
  {doi:10.1111/j.1749-6632.1998.tb09546.x}.

\bibitem{Cramer2016}
Cramer AOJ, van Borkulo CD, Giltay EJ, van~der Maas HLJ, Kendler KS, Scheffer
  M, et~al.
\newblock Major Depression as a Complex Dynamic System.
\newblock PLoS ONE. 2016;11(12):e0167490.
\newblock \href {http://dx.doi.org/10.1371/journal.pone.0167490}
  {doi:10.1371/journal.pone.0167490}.

\bibitem{Borsboom2017}
Borsboom D.
\newblock A Network Theory of Mental Disorders.
\newblock World Psychiatry. 2017;16(1):5-13.
\newblock \href {http://dx.doi.org/10.1002/wps.20375} {doi:10.1002/wps.20375}.

\bibitem{BorsboomCramer2013}
Borsboom D, Cramer AOJ.
\newblock Network Analysis: An Integrative Approach to the Structure of
  Psychopathology.
\newblock Annual Review of Clinical Psychology. 2013;9:91-121.
\newblock \href {http://dx.doi.org/10.1146/annurev-clinpsy-050212-185608}
  {doi:10.1146/annurev-clinpsy-050212-185608}.

\bibitem{bonannochengalatzer-levy2023-resilience-to-potential-trauma-and-adversity-through-regulatory-flexibility}
Bonanno GA, Chen S, Galatzer-Levy IR.
\newblock Resilience to potential trauma and adversity through regulatory
  flexibility.
\newblock Nature Reviews Psychology. 2023;2:663-75.
\newblock \href {http://dx.doi.org/10.1038/s44159-023-00233-5}
  {doi:10.1038/s44159-023-00233-5}.

\bibitem{rueffer-schoenlein-rsmodel-software}
R{\"u}ffer BS, Sch{\"o}nlein M. \texttt{rsmodel}: A Resilience--Symptom
  Dynamical Model of Depression; 2026.
\newblock Source code: \url{https://github.com/bjoseru/rsmodel}.
\newblock Software, Zenodo.
\newblock Available from: \url{https://rsmodel.org} [cited 2026 August 19].
  \href {http://dx.doi.org/10.5281/zenodo.22018288}
  {doi:10.5281/zenodo.22018288}.

\bibitem{isds}
Gr\"une L.
\newblock Input-to-state dynamical stability and its Lyapunov function
  characterization.
\newblock {IEEE} {T}ransactions on {A}utomatic {C}ontrol. 2002;47(9):1499-504.
\newblock \href {http://dx.doi.org/10.1109/TAC.2002.802761}
  {doi:10.1109/TAC.2002.802761}.

\end{thebibliography}
\end{document}